\documentclass[smallextended,nospthms,envcountsect]{svjour3}

\smartqed 

\usepackage{graphicx}
\usepackage{mathptmx}

\usepackage{amssymb}
\usepackage{amsmath}
\usepackage[all,knot,cmtip]{xy}
\xyoption{arc}
\entrymodifiers={+!!<0pt,\fontdimen22\textfont2>} % This will axis
\usepackage{units}
\usepackage{enumitem}
\setenumerate[1]{leftmargin=2em}
\usepackage{mathrsfs}
\usepackage{amsthm}

\numberwithin{equation}{section}

\theoremstyle{plain}   %% This is the default, anyway
\newtheorem{bigtheorem}{Theorem}   % Numbered separately, as A, B, etc.
\newtheorem{theorem}[equation]{Theorem}  % Numbered with the equation counter
\newtheorem{corollary}[equation]{Corollary}     
\newtheorem{lemma}[equation]{Lemma}         
\newtheorem{proposition}[equation]{Proposition} 
\newtheorem{addendum}[equation]{Addendum}

\theoremstyle{definition}
\newtheorem{definition}[equation]{Definition}

\theoremstyle{remark}
\newtheorem{remark}[equation]{Remark}
\newtheorem{example}[equation]{Example}

\newcommand{\Aut}{\operatorname{Aut}}
\newcommand{\Hom}{\operatorname{Hom}}
\newcommand{\map}{\operatorname{map}}
\newcommand{\Map}{\operatorname{Map}}
\newcommand{\End}{\operatorname{End}}
\newcommand{\Ext}{\operatorname{Ext}}
\newcommand{\Gal}{\operatorname{Gal}}
\newcommand{\im}{\operatorname{im}}
\newcommand{\TF}{\operatorname{TF}}
\newcommand{\TC}{\operatorname{TC}}
\newcommand{\TR}{\operatorname{TR}}
\newcommand{\TP}{\operatorname{TP}}
\newcommand{\THH}{\operatorname{THH}}
\newcommand{\HH}{\operatorname{HH}}
\newcommand{\N}{\mathbb{N}}
\newcommand{\Z}{\mathbb{Z}}
\newcommand{\Q}{\mathbb{Q}}
\newcommand{\C}{\mathbb{C}}
\newcommand{\Zp}{\mathbb{Z}_p}

\newcommand{\Fp}{\mathbb{F}_p}
\newcommand{\op}{{\operatorname{op}}}
\newcommand{\holim}{\operatornamewithlimits{holim}}
\newcommand{\ob}{\operatorname{ob}}
\newcommand{\id}{\operatorname{id}}
\newcommand{\pr}{\operatorname{pr}}
\newcommand{\tr}{\operatorname{tr}}
\newcommand{\can}{\operatorname{can}}
\newcommand{\Trd}{\operatorname{Trd}}
\newcommand{\Ird}{\operatorname{Ird}}
\newcommand{\Nrd}{\operatorname{Nrd}}
\newcommand{\Tr}{\operatorname{Tr}}
\newcommand{\Ch}{\mathsf{Ch}}
\newcommand{\tor}{\mathrm{tor}}
\newcommand{\dv}{\mathrm{div}}

\begin{document}

\title{On the $K$-theory of division algebras over local fields}

\titlerunning{$K$-theory of division algebras}

\author{Lars Hesselholt \and Michael Larsen \and Ayelet \rlap{Lindenstrauss}}
%\authorrunning{}
\institute{Lars Hesselholt \at
Nagoya University, Nagoya, Japan and University of Copenhagen,
Copenhagen, Denmark\\
\email{larsh@math.nagoya-u.ac.jp}
\and
Michael Larsen \at 
Indiana University, Bloomington, Indiana\\
\email{mjlarsen@indiana.edu}
\and
Ayelet Lindenstrauss \at
Indiana University, Bloomington, Indiana\\
\email{alindens@indiana.edu}
}

%\keywords{}
%\subjclass[2000]{}
%\thanks{}

\dedication{Andrei Suslin in memoriam}

\date{}

\maketitle

\section*{Introduction}

Let $K$ be a complete discrete valuation field with finite residue
field of characteristic $p$, and let $D$ be a central division algebra
over $K$ of finite index $d$. Thirty years ago, Suslin and
Yufryakov~\cite[Theorem~3]{suslinyufryakov} showed that for all prime
numbers $\ell \neq p$ and integers $j \geqslant 1$, there exists an
isomorphism of $\ell$-adic $K$-groups
\vspace{-1mm}
$$\begin{xy}
(-14,0)*+{ K_j(D,\Z_{\ell}) }="1";
(14,0)*+{ K_j(K,\Z_{\ell}) }="2";
{ \ar^-{\Nrd_{D/K}} "2";"1";};
\end{xy}$$
such that $d \cdot \Nrd_{D/K}$ is equal to the norm homomorphism
$N_{D/K}$. The purpose of this paper is to prove the following
analogous result for the $p$-adic $K$-groups.

\begin{bigtheorem}\label{thm:theoremA}Let $D$ be a central division
algebra of finite index $d$ over a complete discrete valuation field
$K$ with finite residue field of odd characteristic $p$. For all
integers $j \geqslant 1$, there exists a canonical isomorphism of
$p$-adic $K$-groups
\vspace{-1mm}
$$\begin{xy}
(-14,0)*+{ K_j(D,\Zp) }="1";
(14,0)*+{ K_j(K,\Zp) }="2";
{ \ar^-{\Nrd_{D/K}} "2";"1";};
\end{xy}$$
such that $d \cdot \Nrd_{D/K}$ is equal to the norm homomorphism $N_{D/K}$.
\end{bigtheorem}

By contrast to the norm homomorphism $N_{D/K}$, we do not know that
the reduced norm isomorphism $\Nrd_{D/K}$ is induced by a map of
$K$-theory spectra, and such a map may well not exist; compare
Merkurjev~\cite[Proposition~4]{merkurjev}.

The tool that makes it possible to now prove Theorem~\ref{thm:theoremA}  is
the topological cyclic homology introduced by B\"{o}kstedt, Hsiang,
and Madsen~\cite{bokstedthsiangmadsen}. 
The recent work by Nikolaus and Scholze~\cite{nikolausscholze} has
greatly clarified the nature of this theory, and we will use their
setup, which we briefly explain. Let $\smash{
  \mathsf{Sp}_p^{B\,\mathbb{T}} }$ be
the infinity-category of $p$-complete spectra with an action by the
circle $\mathbb{T}$. Given $\smash{ X \in
  \mathsf{Sp}_p^{B\,\mathbb{T}} }$,
we write
\vspace{-1mm}
$$\begin{xy}
(0,0)*+{ \TC^{-}(X) }="1";
(23,0)*+{ \TP(X) }="2";
{ \ar^-{\can} "2";"1";};
\end{xy}$$
for the canonical map from the homotopy fixed points spectrum to the
Tate spectrum of $X$. We refer to these spectra as the negative
topological cyclic homology and the periodic topological cyclic
homology of $X$, respectively. A cyclotomic structure on $X$ is a map
of spectra with $\mathbb{T}$-action
\vspace{-1mm}
$$\begin{xy}
(0,0)*+{ \phantom{X^{tC_p}} X }="1";
(20,0)*+{ X^{tC_p} \phantom{X} }="2";
{ \ar@<-.3ex>^-{\varphi} "2";"1";};
\end{xy}$$
called the Frobenius map. The target of this map is the Tate
spectrum of $X$ with respect to the subgroup $C_p \subset \mathbb{T}$
of order $p$. It has a residual $\mathbb{T}/C_p$-action, which we
consider a $\mathbb{T}$-action via the $p$th root isomorphism
$\rho \colon \mathbb{T} \to \mathbb{T}/C_p$. The Frobenius induces a
map of homotopy $\mathbb{T}$-fixed points,
\vspace{-1mm}
$$\begin{xy}
(0,0)*+{ \TC^{-}(X) = X^{h\mathbb{T}} }="1";
(45,0)*+{ (X^{tC_p})^{h(\mathbb{T}/C_p)} \simeq X^{t\mathbb{T}} =
  \TP(X), }="2";
{ \ar^-{\varphi^{h\mathbb{T}}} "2";"1";};
\end{xy}$$
the target of which is canonically identified with the
$\mathbb{T}$-Tate spectrum by the Tate orbit lemma of Nikolaus and
Scholze~\cite[Lemma~I.2.1]{nikolausscholze}. We will abuse notation
and write $\varphi \colon \TC^{-}(X) \to \TP(X)$ for the resulting
map, which we again call the Frobenius map. Now, the
topological cyclic homology of $X$ is the homotopy equalizer
\vspace{-1mm}
$$\begin{xy}
(0,0)*+{ \TC(X) }="1";
(22,0)*+{ \TC^{-}(X) }="2";
(44,0)*+{ \TP(X) }="3";
{ \ar^-{i} "2";"1";};
{ \ar@<.8ex>^-{\varphi} "3";"2";};
{ \ar@<-.8ex>_-{\can} "3";"2";};
\end{xy}$$
of the Frobenius map and the canonical map. Nikolaus and
Scholze also show that, in a natural way, the $p$-complete
cyclotomic spectra can be organized into a symmetric monoidal stable
infinity-category $\mathsf{CycSp}_p$, whose tensor unit is the
$p$-complete sphere spectrum $\mathbb{S}_p$ with trivial
$\mathbb{T}$-action and with the composite map
\vspace{-1mm}
$$\begin{xy}
(0,0)*+{ \phantom{\mathbb{S}_p^{tC_p}} \mathbb{S}_p }="1";
(20,0)*+{ \mathbb{S}_p^{hC_p} }="2";
(40,0)*+{ \mathbb{S}_p^{tC_p} \phantom{\mathbb{S}_p} }="3";
{ \ar@<-.2ex>^-{\tilde{\varphi}} "2";"1";};
{ \ar@<-.2ex>^-{\can} "3";"2";};
\end{xy}$$
as its cyclotomic structure map. Here, the left-hand map is the
canonical map from a spectrum with trivial action to its fixed
points. Finally, the left-hand map in the homotopy equalizer above is
canonically identified with the ``forgetful'' map
\vspace{-1mm}
$$\begin{xy}
(0,0)*+{ \map_{\mathsf{CycSp}_p}(\mathbb{S}_p,X) }="1";
(35,0)*+{ \map_{\mathsf{Sp}_p^{B\,\mathbb{T}}}(\mathbb{S}_p,X) }="2";
{ \ar@<.4ex>^-{i} "2";"1";};
\end{xy}$$
between mapping spectra in the stable infinity-categories of
$p$-complete cyclotomic spectra and $p$-complete spectra with
$\mathbb{T}$-action, respectively.

We now let $K$ be the quotient field
of a complete discrete valuation ring $S$ with finite residue field
$k_S$ of characteristic $p$, let $D$ be a central division algebra
over $K$, and let $d = \dim_K(D)^{\nicefrac{1}{2}}$ be the index of
$D$ over $K$. The valuation on $K$ extends uniquely to a valuation on
$D$, and the subring $A \subset D$ of elements of non-negative
valuation is the unique maximal $S$-order. In~\cite{hm4}, the first
author and Madsen produced a $p$-complete cyclotomic spectrum $\THH(S
\,|\, K,\Zp)$ and a trace map
\vspace{-1mm}
$$\begin{xy}
(0,0)*+{ K(K,\Zp) }="1";
(29,0)*+{ \TC(S \,|\, K,\Zp) }="2";
{ \ar^-{\tr} "2";"1";};
\end{xy}$$
to its topological cyclic homology spectrum, which we abbreviate as
indicated, from the $p$-completion of the algebraic $K$-theory
spectrum of the field $K$. This construction also defines a
$p$-complete cyclotomic spectrum $\THH(A\,|\,D,\Zp)$ and a trace map
\vspace{-1mm}
$$\begin{xy}
(0,0)*+{ K(D,\Zp) }="1";
(29,0)*+{ \TC(A \,|\, D,\Zp) }="2";
{ \ar^-{\tr} "2";"1";};
\end{xy}$$
to its topological cyclic homology spectrum from the $p$-completion of
the algebraic $K$-theory spectrum of the division ring $D$. Moreover,
by~\cite[Theorem~D]{hm}, both maps induce isomorphisms of homotopy
groups in degrees $j \geqslant 1$. Hence, Theorem~\ref{thm:theoremA} is a
consequence of the following more precise
Theorem~\ref{thm:theoremB}. In order to state it, we first remark that
$\THH(S\,|\,K,\Zp)$ has a canonical structure of an
$\mathbb{E}_{\infty}$-algebra in $\mathsf{CycSp}_p$, that
$\THH(A\,|\,D,\Zp)$ has a canonical structure of a
$\THH(S\,|\,K,\Zp)$-module, and that the trace map $\Tr_{A/S}$ is a
map of $\THH(S\,|\,K,\Zp)$-modules
\vspace{-1mm}
$$\begin{xy}
(0,0)*+{ \THH(A\,|\,D,\Zp) }="1";
(34,0)*+{ \THH(S\,|\,K,\Zp) }="2";
{ \ar^-{\Tr_{A/S}} "2";"1";};
\end{xy}$$
in $\mathsf{CycSp}_p$. The following is our main result.

\begin{bigtheorem}\label{thm:theoremB}Let $K$ be the quotient field of
a complete discrete valuation ring $S$ with finite residue field of odd
characteristic $p$, let $D$ be a central division algebra over $K$,
and let $A \subset D$ be the maximal $S$-order.
\begin{enumerate}
\item[{\rm (1)}]There exists an equivalence of \hskip1.5pt
$\THH(S\,|\,K,\Zp)$-modules in cyclotomic spectra,
\vspace{-1mm}
$$\begin{xy}
(0,0)*+{ \THH(A\,|\,D,\Zp) }="1";
(35,0)*+{ \THH(S\,|\,K,\Zp). }="2";
{ \ar^-{\Trd_{A/S}} "2";"1";};
\end{xy}$$
\item[{\rm (2)}]As maps of \hskip1.5pt $\THH(S\,|\,K,\Zp)$-modules in
  spectra with $\mathbb{T}$-action,
$$d \cdot \Trd_{A/S} \,\simeq\, \Tr_{A/S}.$$
\item[{\rm (3)}]If $p$ divides $d$, then, as maps of \hskip1.5pt
  $\THH(S\,|\,K,\Zp)$-modules in cyclotomic spectra,
$$d \cdot \Trd_{A/S} \,\not\simeq\, \Tr_{A/S}.$$
\end{enumerate}
\end{bigtheorem}

The first part of the theorem shows, in particular, that
$\TC_*(A\,|\,D,\Zp)$ is free on a generator of degree $0$ as a graded
$\TC_*(S\,|\,K,\Zp)$-module, and this implies the first statement in
Theorem~\ref{thm:theoremA}. This generator is not in the image of the
cyclotomic trace, and $K_*(D,\Zp)$ is neither free nor
finitely generated as a graded $K_*(K,\Zp)$-module. 

To produce the desired equivalence of $\THH(S\,|\,K,\Zp)$-modules in
$\mathsf{CycSp}_p$, we instead produce its inverse equivalence
\vspace{-1mm}
$$\begin{xy}
(0,0)*+{ \THH(S\,|\,K,\Zp) }="1";
(35,0)*+{ \THH(A\,|\,D,\Zp). }="2";
{ \ar^-{\Ird_{A/S}} "2";"1";};
\end{xy}$$
The space of such maps has group of components $\TC_0(A\,|\,D,\Zp)$,
and, similarly, the corresponding space of maps in $\smash{
  \mathsf{Sp}_p^{B\,\mathbb{T}} }$ has group of components
$\TC_0^{-}(A\,|\,D,\Zp)$. To understand $\TC_*^{-}(A\,|\,D,\Zp)$, we
choose a maximal unramified subfield $K \subset L \subset D$ and let
$S \subset T \subset A$ be the subring of elements of non-negative
valuation. The extension $L/K$ is of degree $d$ and the Galois group
$G$ of $L/K$ is canonically isomorphic to that of the extension
$k_T/k_S$ of residue fields. In general, for $R$ a unital associative
ring, we write $\mathcal{P}_R$ for the exact category of finitely
generated projective left $R$-modules. In the case at hand, the ring
homomorphisms
$$\xymatrix{
{ A \otimes_ST } &
{ T \otimes_ST } \ar[l]_-(.43){\pi} \ar[r]^-{\delta} &
{ T } \cr
}$$
given by the canonical inclusion and the multiplication are finite
locally free, and hence, we have the functors $\pi_* \colon
\mathcal{P}_{A\otimes_ST} \to \mathcal{P}_{T\otimes_ST}$ and
$\delta_* \colon \mathcal{P}_{T} \to \mathcal{P}_{T\otimes_ST}$ given
by restriction-of-scalars along $\pi$ and $\delta$, respectively. Let
$\pi^!$ be the right adjoint of $\pi_*$ given by
coextension-of-scalars along $\pi$, and let $\delta^*$ be the left
adjoint of $\delta_*$ given by extension-of-scalars along $\delta$. We
consider the diagram
$$\begin{xy}
(0,7)*+{ \mathcal{P}_{A} }="11";
(25,7)*+{ \mathcal{P}_{A\otimes_ST} }="12";
(0,-7)*+{ \mathcal{P}_S }="21";
(25,-7)*+{ \mathcal{P}_T, }="22";
{ \ar^-{f^*} "12";"11";};
{ \ar^-{f^*} "22";"21";};
{ \ar@<-.7ex>_-{\Trd_{A\otimes_ST/T}} "22";"12";};
{ \ar@<-.7ex>_-{\Ird_{A\otimes_ST/T}} "12";"22";};
\end{xy}$$
where the right-hand adjunction is the composite adjunction
$(\delta^* \circ \pi_*,\pi^! \circ \delta_*)$, and where the
horizontal functors are extension-of-scalars along the canonical
inclusion $f \colon S \to T$. Said adjunction is not an adjoint
equivalence of categories, but it becomes one after
extension-of-scalars along the canonical inclusion $h \colon S \to K$,
exhibiting the well-known Morita equivalence of $D \otimes_KL$ and
$L$. The ring $A \otimes_ST$ is not a maximal $T$-order in
$D \otimes_KL$, so the following result, which we prove in
Section~\ref{sec:cats}, came as a rather fortunate surprise. 

\begin{bigtheorem}\label{thm:theoremC}The ring $A \otimes_ST$ is left
regular.
\end{bigtheorem}

Here, we follow Bass~\cite[p.~122]{bass} and call a ring $R$ left
regular if every finitely generated left $R$-module admits a finite 
resolution by finitely generated projective left
$R$-modules. Using Theorem~\ref{thm:theoremC}, we show that, in the
diagram
$$\begin{xy}
(0,7)*+{ \TC_*^{-}(A\,|\,D,\Zp) }="11";
(47,7)*+{ H^0(G,\TC_*^{-}(A \otimes_ST \,|\, D \otimes_KL,\Zp)) }="12";
(0,-7)*+{ \TC_*^{-}(S\,|\,K,\Zp) }="21";
(47,-7)*+{ H^0(G,\TC_*^{-}(T\,|\,L,\Zp)), \hspace{3mm} }="22";
{ \ar^-{f^*} "12";"11";};
{ \ar^-{f^*} "22";"21";};
{ \ar@<-1.4ex>_-{\Trd_{A \otimes_ST/T}} "22";"12";};
{ \ar_-{\Ird_{A \otimes_ST/T}} "12";"22";};
\end{xy}$$
the horizontal morphisms are isomorphisms, and the Morita equivalence
mentioned above implies that also the vertical morphisms are
isomorphisms. All morphisms in the diagram are graded
$\TC_*^{-}(S\,|\,K,\Zp)$-module homomorphisms, so we conclude that
the graded $\TC_*^{-}(S\,|\,K,\Zp)$-module $\TC_*^{-}(A\,|\,D,\Zp)$ is
free on a single generator of degree zero. We let
$y \in \TC_0^{-}(A\,|\,D,\Zp)$ be the unique generator with the
property that $f^*(y) = \Ird_{A \otimes_ST/T}(f^*(1))$. It satisfies that
$\varphi(y) = \can(y)$, and accordingly, there exists
$\tilde{y} \in \TC_0(A\,|\,D,\Zp)$ with $i(\tilde{y}) = y$. This
implies part~(1) of Theorem~\ref{thm:theoremB}, and part~(2) follows
from the fact that $\smash{ \Tr_{A/S}(y) = d \cdot 1 }$. This equation
also characterizes the generator $y$, since the common groups
$\TC_0^{-}(A\,|\,D,\Zp)$ and $\TC_0^{-}(S\,|\,K,\Zp)$ are free
$\Zp$-modules of rank one. By contrast, the common groups
$\TC_0(A\,|\,D,\Zp)$ and $\TC_0(S\,|\,K,\Zp)$ are free $\Zp$-modules
of rank two, and we show that if $p$ divides $d$, then it is not
possible to choose the generator $\tilde{y}$ such that
$\Tr_{A/S}(\tilde{y})$ is a divisible by $d$, which implies part~(3)
of Theorem~\ref{thm:theoremB}.

The uniqueness of the generator $y \in \TC_0^{-}(A\,|\,D,\Zp)$ with
$\Tr_{A/S}(y) = d \cdot 1$ implies that the reduced trace isomorphisms
on negative topological cyclic homology groups and on periodic
cyclic homology groups are canonical and satisfy
$$d \cdot \Trd_{A/S} = \Tr_{A/S}.$$
That the corresponding statements hold on topological cyclic homology
groups is not immediately clear and is false in degree zero, if $p$
divides $d$. However, for $j \geqslant 1$ and odd, we show that there
are exact sequences
$$\begin{xy}
(0,7)*+{ 0 }="11";
(0,-7)*+{ 0 }="21";
(22,7)*+{ \TC_j(A\,|\,D,\Zp) }="12";
(22,-7)*+{ \TC_j(S\,|\,K,\Zp) }="22";
(55,7)*+{ \TC_j^{-}(A\,|\,D,\Zp) }="13";
(55,-7)*+{ \TC_j^{-}(S\,|\,K,\Zp) }="23";
(90,7)*+{ \TP_j(A\,|\,D,\Zp)\phantom{,} }="14";
(90,-7)*+{ \TP_j(S\,|\,K,\Zp), }="24";
{ \ar "12";"11";};
{ \ar^-{i} "13";"12";};
{ \ar^-{\varphi - \can} "14";"13";};
{ \ar^-{\Trd_{A/S}} "22";"12";};
{ \ar^-{\Trd_{A/S}} "23";"13";};
{ \ar^-{\Trd_{A/S}} "24";"14";};
{ \ar "22";"21";};
{ \ar^-{i} "23";"22";};
{ \ar^-{\varphi - \can} "24";"23";};
\end{xy}$$
which show that also the left-hand map $\Trd_{A/S}$ is canonical and
that its $d$th multiple is equal to $\Tr_{A/S}$. Similarly, for
$j \geqslant 2$ and even, there are exact sequences
$$\begin{xy}
(0,7)*+{ \TC_{j+1}^{-}(A\,|\,D,\Zp) }="11";
(0,-7)*+{ \TC_{j+1}^{-}(S\,|\,K,\Zp) }="21";
(38,7)*+{ \TP_{j+1}(A\,|\,D,\Zp)\phantom{,} }="12";
(38,-7)*+{ \TP_{j+1}(S\,|\,K,\Zp), }="22";
(73,7)*+{ \TC_j(A\,|\,D,\Zp) }="13";
(73,-7)*+{ \TC_j(S\,|\,K,\Zp) }="23";
(95,7)*+{ 0 }="14";
(95,-7)*+{ 0 }="24";
{ \ar^-{\varphi-\can} "12";"11";};
{ \ar^-{\partial} "13";"12";};
{ \ar "14";"13";};
{ \ar^-{\Trd_{A/S}} "21";"11";};
{ \ar^-{\Trd_{A/S}} "22";"12";};
{ \ar^-{\Trd_{A/S}} "23";"13";};
{ \ar^-{\varphi-\can} "22";"21";};
{ \ar^-{\partial} "23";"22";};
{ \ar "24";"23";};
\end{xy}$$
which show that the right-hand map $\Trd_{A/S}$ is canonical and that
its $d$th multiple is equal to $\Tr_{A/S}$. This proves that latter statement in
Theorem~\ref{thm:theoremA} that the reduced norm isomorphism is
canonical and satisfies $d \cdot \Nrd_{D/K} = N_{D/K}$.

Our proofs of Theorem~\ref{thm:theoremB}~(2)--(3)
use~\cite[Theorem~5.5.1]{hm4}, which is the reason that we assume 
$p$ to be odd. The remaining results hold also for $p = 2$, as do
all our results, if $S$ is of equal characteristic. We expect our results
to hold also for $p = 2$.

\begin{acknowledgements}We gratefully acknowledge the generous
assistance that we have received from a DNRF Niels Bohr Professorship,
Simons Foundation Grant 359565, and NSF Grants DMS-1702152 and
DMS-1552766. The first author also thanks Indiana University and the
Hausdorff Research Institute for Mathematics in Bonn for their
hospitality and support, and the second and third author thank the
University of Copenhagen for its hospitality and support. The first
author further thanks Thomas Geisser for helpful discussions. We 
are much indebted to Jacob Lurie for pointing out a mistake in an
earlier version of the work presented here, and finally, we thank an
anonymous referee for many very helpful remarks.
\end{acknowledgements}

\section{Categories of modules}\label{sec:cats}

In this section, we examine the structure of the category of left
modules over the ring $A \otimes_ST$ and prove
Theorem~\ref{thm:theoremC} of the introduction.

If $R$ is a unital associative ring $R$, then we write
$\mathsf{Mod}_R$ for the category of left $R$-modules, and
if $f \colon R \to S$ is a ring homomorphism, then we define the 
restriction along $f$ to be the functor $f_* \colon \mathsf{Mod}_S \to
\mathsf{Mod}_R$ that to an $S$-module $M$ assigns the $R$-module
$f_*(M)$ with the same underlying additive group as that of $M$ and
with left scalar multiplication by $a \in R$ given by left scalar
multiplication by $f(a) \in S$ and that to an $S$-linear map $h \colon
M \to M'$ assigns the same map $h \colon f_*(M) \to f_*(M')$. We note
that the restriction of scalars along the identity homomorphism
$\id_R$ and the identity functor of $\mathsf{Mod}_R$ are canonically
naturally isomorphic, as are $(g \circ f)_*$ and $f_* \circ g_*$ for
composable ring homomorphisms $f \colon R \to S$ and
$g \colon S \to T$.

The functor $f_*$ admits both a
left adjoint functor and a right adjoint functor. We say that a choice
of an adjunction $(f^*,f_*,\epsilon,\eta)$ from $\mathsf{Mod}_R$ to
$\mathsf{Mod}_S$ is an extension of scalars along $f$; and we say that
a choice of adjunction $(f_*,f^!,\epsilon,\eta)$ from $\mathsf{Mod}_R$
to $\mathsf{Mod}_S$ is a coextension of scalars along 
$f$. If $f \colon R \to S$ and $g \colon S \to T$ are composable ring
homomorphisms, then the functors $g^* \circ f^*$ and $(g \circ f)^*$
and the functors $g^! \circ f^!$ and $(g \circ f)^!$ are canonically
naturally isomorphic, and the extension and coextension along the
identity homomorphism $\id_R$ both are canonically naturally
isomorphic to the identity functor;
compare~\cite[Theorem~IV.7.2]{maclane}.

We write $\mathcal{M}_R$ and $\mathcal{P}_R$ for the full
subcategories of $\mathsf{Mod}_R$ whose objects are the finitely
generated left $R$-modules and the finitely generated projective left
$R$-modules, respectively. Let $f \colon R \to S$ be a ring
homomorphism. The extension of scalars along $f$ restricts to functors
$f^* \colon \mathcal{M}_R \to \mathcal{M}_S$ and
$f^* \colon \mathcal{P}_R \to \mathcal{P}_S$, the former of which is
an exact functor if and only if $f$ is flat; the restriction of
scalars along $f$ restricts to a functor
$f_* \colon \mathcal{M}_S \to \mathcal{M}_R$, if $f$ is finite, and to
a functor $f_* \colon \mathcal{P}_S \to \mathcal{P}_R$, if $f$ is
finite and if $S$ considered as a left $R$-module via $f$ is
projective, both of  which are exact; and the coextension of scalars
along $f$ restricts to exact functors
$f^! \colon \mathcal{M}_R \to \mathcal{M}_S$, if $S$ is a finitely generated projective $R$-module, and
$f^! \colon \mathcal{P}_R \to \mathcal{P}_S$, if, in addition, the coextension of $R$, $\Hom_R(S,R)$, is a
finitely generated  projective $S$-module.  In particular, if $S$ is a finitely generated projective $R$-module
and every object $M$ of $\mathcal{M}_S$ such that $f_*M$ is an object of $\mathcal{P}_R$ must in fact
lie in $\mathcal{P}_S$, then $f^!\colon \mathcal{P}_R \to \mathcal{P}_S$ exists and is exact.

We again let $S$ be a complete discrete valuation
ring with finite residue field $k_S$ of characteristic $p$ and with
quotient field $K$, and let $D$ be a finite dimensional central
division algebra over $K$. We recall the structure of $D$
following~\cite[Chapter~3]{reiner}. The valuation $v_K$ on $K$ extends
uniquely to a discrete valuation $v_D$ on $D$ given by
$$v_D(x) = \frac{1}{\dim_K(D)}v_K(N_{D/K}(x)),$$
where $x \in D^*$ and $N_{D/K} \colon D^* \to K^*$ is the norm. The
algebra $D$ is complete with respect to $v_D$, and the subring
$A \subset D$ of elements of non-negative valuation is both the
integral closure of $S$ in $D$ and the unique maximal $S$-order in
$D$. We choose a maximal subfield $K \subset L \subset D$ with the property
that the extension $L/K$ is unramified and let $T \subset L$ be the
integral closure of $S$. The equality $\dim_K(L) = \dim_L(D)$ holds,
and the common dimension $d$ is called the index of $D$ over
$K$. Hence, if $\pi_D$ is a generator of the unique maximal ideal
$\mathfrak{m}_D \subset A$, then the tuple
$(1,\pi_D,\dots,\pi_D^{d-1})$ is a basis of $D$ as a left $L$-vector
space. Now, by~\cite[Theorem~14.5]{reiner}, we may choose the
generator $\pi_D$ such that $\pi_D^d$ is contained in $S$ (and hence
is a generator $\pi_K$ of the maximal ideal $\mathfrak{m}_K \subset
S$) and such that the inner automorphism $x\mapsto \pi_D^{\phantom{-}}x\pi_D^{-1}$
of $D$ restricts to an automorphism $\sigma$
of $L/K$ which generates $\Gal(L/K)$. The map
$$\begin{xy}
(-12,0)*+{ \Gal(L/K) }="1";
(12,0)*+{ \Gal(k_T/k_S) }="2";
{ \ar^-{\epsilon} "2";"1";};
\end{xy}$$
defined by $\epsilon(g)(y+\mathfrak{m}_T)=g(y)+\mathfrak{m}_T$
is an isomorphism, since $L/K$ is unramified. It maps the generator
$\sigma$ of the domain to a generator of the target, which we may
therefore write as the $r$th power of the Frobenius automorphism, for a unique
integer $0 < r < d$ relatively prime to $d$. The class of
$r/d$ in $\Q/\Z$ is called the Hasse invariant of $D$.  It determines the
central division $K$-algebra $D$,
up to non-canonical isomorphism.  Moreover, every element of $\Q/\Z$
occurs as the index of some central
division $K$-algebra $D$.

Let $k$ be a commutative ring, let $R$ be a commutative $k$-algebra,
and let $\varphi \colon R \to R$ be a $k$-algebra automorphism. The
twisted polynomial algebra $R^{\varphi}\{x\}$ is the quotient of the
coproduct $R *_k k[x]$, in the category of unital associative $k$-algebras, of $R$
and $k[x]$ by the two-sided ideal generated by the family of elements
$\varphi(a) x - xa$ with $a \in R$. We let
$i_1 \colon R \to R^{\varphi}\{x\}$ and $i_2 \colon k[x] \to
R^{\varphi}\{x\}$ be the two $k$-algebra homomorphisms defined as the
compositions of the respective canonical inclusions into the coproduct
followed by the canonical projection. 

\begin{lemma}\label{lem:basechange}In the situation above, if $k'$ is a
commutative $k$-algebra and $R'=R\otimes _k k'$, then there is a unique
isomorphism of $k'$-algebras
\vspace{-1mm}
$$\begin{xy}
(0,0)*+{ (R' )^{\sigma\otimes \id}\{ x\} }="1";
(29,0)*+{ R^\sigma \{ x\} \otimes_k k' }="2";
{ \ar^-{u} "2";"1";};
\end{xy}$$
compatible with the maps $i_1$ and $i_2$ over $k$ and $k'$,
respectively.
\end{lemma}

\begin{proof}The universal property of coproducts gives a
canonical map of $k'$-algebras
\vspace{-1mm}
$$\begin{xy}
(0,0)*+{ R' *_{k'} k'[x] }="1";
(29,0)*+{ (R *_k k[x])\otimes_k k', }="2";
{ \ar "2";"1";};
\end{xy}$$
the universal property of extension of scalars gives a
canonical map of $k$-modules
\vspace{-1mm}
$$\begin{xy}
(0,0)*+{ (R *_k k[x])\otimes_k k' }="1";
(29,0)*+{ R' *_{k'} k'[x], }="2";
{ \ar "2";"1";};
\end{xy}$$
and the two maps are mutually inverse. Moreover, the kernels of the
canonical projections to $(R' )^{\sigma\otimes \id}\{ x\}$ and
$R^\sigma \{ x\} \otimes_k k' $, respectively, are identified under
these isomorphisms. 
\end{proof}

\begin{lemma}\label{lem:twisty}With notation as above, let $v$ be the unique
$S$-algebra homomorphism
\vspace{-1mm}
$$\begin{xy}
(0,0)*+{ T^\sigma\{x\} }="1";
(19,0)*+{ A }="2";
{ \ar^-{v} "2";"1";};
\end{xy}$$
such that $v \circ i_1 \colon T \to A$ is the canonical inclusion and
$v \circ i_2 \colon S[x] \to A$ is the unique $S$-algebra homomorphism
mapping $x$ to $\pi_D$. Then $v$ is surjective, and its kernel is the
two-sided ideal generated by $x^d-\pi_K$.
\end{lemma}

\begin{proof}To prove surjectivity, recall that every element
$a \in D$ can be written uniquely as an $L$-linear combination 
$a = y_0 + \dots + y_{d-1}\pi_D^{d-1}$. Now, since $v_D(L^\times) = d\Z$,
the values $v_D(y_i\pi_D^i)$ are pairwise distinct. Therefore, we
conclude from the ultrametric inequality that $a \in A$ if and only if
$y_0,\dots,y_{d-1} \in T$ as desired.  Clearly, $x^d - \pi_K$ lies in
the kernel of $v$, and by the linear independence of
$(1,\pi_D, \dots, \pi_D^{d-1})$ over $L$, no polynomial of lower
degree does so. Therefore, by the right division algorithm, every
element in the kernel of $v$ is a left multiple of the (central)
element $x^d-\pi_K$.
\end{proof}

\begin{corollary}\label{lem:moretwisty}The unique $T$-algebra
homomorphism
\vspace{-1mm}
$$\begin{xy}
(0,0)*+{ (T \otimes_ST)^{\sigma \otimes \id}\{x\} }="1";
(31,0)*+{ A\otimes_S T }="2";
{ \ar^-{v'} "2";"1";};
\end{xy}$$
such that
$v' \circ (i_1\otimes \id) \colon T \otimes_ST \to A \otimes_ST$ is the
canonical inclusion and such that
$v' \circ( i_2\otimes\id) \colon T[x] \to A \otimes_ST$ the unique
$T$-algebra homomorphism that maps $x$ to $\pi_D \otimes 1$
is surjective, and its kernel is the two sided ideal $(x^d-\pi_K\otimes 1)$.
\end{corollary}

\begin{proof}The map $v'$ factors as the composition
\vspace{-1mm}
$$\begin{xy}
(0,0)*+{ (T \otimes_ST)^{\sigma \otimes \id}\{x\} }="1";
(33,0)*+{ T^{\sigma}\{x\} \otimes_S T }="2";
(61,0)*+{ A\otimes_S T }="3";
{ \ar^-{u} "2";"1";};
{ \ar^-{v \otimes \id} "3";"2";};
\end{xy}$$
of the isomorphism in Lemma~\ref{lem:twisty} and the extension of scalars
along $f \colon S \to T$ of the isomorphism in
Lemma~\ref{lem:moretwisty}. 
\end{proof}

We define a category $\mathsf{Mod}_{\,T,G}$ as follows. An object is
a triple $(N,(N_g)_{g \in G},\varphi)$, where $N$ is a left
$T$-module, $(N_g)_{g \in G}$ is a grading on $N$ of type $G = \Gal(L/K)$,  
and $\varphi \colon N \to N$ is a graded $T$-linear endomorphism of
degree $\sigma^{-1} \in G$ such that $\varphi^{d}$ is
equal to left multiplication by $\pi_K \in T$. A morphism
$h \colon (N,(N_g),\varphi) \to (N',(N_g'),\varphi')$ is a
graded $T$-linear map $h \colon N \to N'$ of degree $1 \in G$ such
that $h \circ \varphi = \varphi' \circ h$. We write the composition
law in $G$ multiplicatively and refer 
to~\cite[Chapter~2,~\S11]{bourbaki1} for the definitions of graded
modules and graded homomorphisms. We recall that since $f \colon S \to
T$ is faithfully \'{e}tale with Galois group $G$, the ring homomorphism
$$\begin{xy}
(-11,0)*+{ T \otimes_ST}="1";
(11,0)*+{ \prod_{g \in G}T }="2";
{ \ar^-{w} "2";"1";};
\end{xy}$$
with $g$th component 
$w_g(t_1 \otimes t_2) = g(t_1)t_2$ is an isomorphism. We let $(e_g)_{g \in G}$
be the family of orthogonal idempotents in $T \otimes_ST$ such that
$w_g(e_{g'}) = \delta_{g,g'}$ and let
$$\begin{xy}
(-12,0)*+{ \mathsf{Mod}_{A \otimes_ST} }="1";
(12,0)*+{ \mathsf{Mod}_{\,T,G} }="2";
{ \ar^-{F} "2";"1";};
\end{xy}$$
be the functor that to a left $A \otimes_ST$-module $N$ assigns the
triple $(N,(N_g)_{g \in G},\varphi)$, where, by abuse of notation, $N$
is the underlying left $T$-module of the left $A \otimes_ST$-module 
$N$, where $N_g \subset N$ is the $T$-submodule
$e_g \cdot N \subset N$, and where $\varphi \colon N \to N$ is the
$T$-linear map given by left multiplication by
$\pi_D \otimes 1 \in A \otimes_ST$.

\begin{proposition}\label{prop:gradedmodules}The functor
$$\begin{xy}
(-12,0)*+{ \mathsf{Mod}_{A \otimes_ST} }="1";
(12,0)*+{ \mathsf{Mod}_{\,T,G} }="2";
{ \ar^-{F} "2";"1";};
\end{xy}$$
is an equivalence of categories.
\end{proposition}

\begin{proof}First, to prove that $F$ is well-defined, we must verify
that the $T$-linear map $\varphi \colon N \to N$ defined by
$\varphi(y) = \pi_D \otimes 1 \cdot y$ is indeed graded of 
degree $\sigma^{-1}$ and that $\varphi^d$ is given by multiplication
by $\pi_K$. The latter holds, since
$$\pi_D^d \otimes 1 = \pi_K \otimes 1 = 1 \otimes \pi_K,$$
and to prove the former we apply Corollary~\ref{lem:moretwisty} to
conclude that
$$\varphi((t_1 \otimes t_2)y) = (\sigma(t_1) \otimes t_2)\varphi(y),$$
for all $y \in N$ and $t_1 \otimes t_2 \in T \otimes_ST$. Moreover,
the commutative diagram
$$\begin{xy}
(-11,7)*+{ T \otimes_ST}="11";
(11,7)*+{ \prod_{g \in G}T\phantom{,} }="12";
(-11,-7)*+{ T \otimes_ST}="21";
(11,-7)*+{ \prod_{g \in G}T, }="22";
{ \ar^-{w} "12";"11";};
{ \ar^-{w} "22";"21";};
{ \ar^-{\sigma \otimes \id} "21";"11";};
{ \ar^-{(\sigma \otimes \id)^w} "22";"12";};
\end{xy}$$
where the map $(\sigma \otimes \id)^w$ is defined by
$\pr_g \circ (\sigma \otimes \id)^w = \pr_{g \circ \sigma}$,
shows that
$$(\sigma \otimes \id)(e_g) = e_{g \circ \sigma^{-1}}.$$
This shows that $\varphi \colon N \to N$ is graded of degree
$\sigma^{-1}$ as desired.

We define a quasi-inverse functor
$$\begin{xy}
(-12,0)*+{ \mathsf{Mod}_{\,T,G} }="1";
(12,0)*+{ \mathsf{Mod}_{A \otimes_ST} }="2";
{ \ar^-{H} "2";"1";};
\end{xy}$$
as follows. Given an object $(N,(N_g)_{g \in G},\varphi)$ of the
domain category, we first use the grading $(N_g)_{g \in G}$ to define
a left $T \otimes_ST$-module structure on the left $T$-module $N$ by
letting $t_1 \otimes t_2 \in T \otimes_ST$ multiply by
$w_g(t_1 \otimes t_2)$ on $N_g \subset N$. Moreover, since the
$T$-linear map $\varphi \colon N \to N$ is graded of degree
$\sigma^{-1}$, the argument above shows that
$$\varphi((t_1 \otimes t_2)y) = (\sigma(t_1) \otimes t_2)\varphi(y),$$
for all $y \in N$ and $t_1 \otimes t_2 \in T \otimes_ST$. Since, in
addition, $\varphi^d$ is given by multiplication by $\pi_K$, this
defines a $(T \otimes_ST)^{\sigma \otimes   \id}\{x\} / (x^d-\pi_K
\otimes 1)$-module structure on $N$, where the left multiplication by
$x$ is given by the map $\varphi \colon N \to N$, and by
Corollary~\ref{lem:moretwisty}, this defines a left $A \otimes_ST$-module
structure on $N$. This defines the functor $H$, and it is clear that
$F \circ H$ and $H \circ F$ are equal to the respective identity
functors.
\end{proof}

\begin{example}\label{ex:sizeone}Right multiplication on $A \otimes_ST$
by the idempotents $(e_h)_{h \in G}$ defined in the proof of
Proposition~\ref{prop:gradedmodules} gives rise to  
a direct sum decomposition
$$A \otimes_ST = \bigoplus_{h \in G} \, A \otimes_ST \cdot e_h$$
as left $A \otimes_ST$-modules. Hence, as a left $A
\otimes_ST$-module, each of the $d$ summands is projective. We now
evaluate $F(A \otimes_ST \cdot e_h)$.  By
Corollary~\ref{lem:moretwisty}, 
$$(\pi_D\otimes 1) \cdot e_h = (\sigma
\otimes  1)(e_h) \cdot (\pi_D\otimes 1) = e_{h \circ \sigma^{-1}} \cdot (\pi_D\otimes 1),$$
so as a left $T$-submodule of $A \otimes_ST$,
$$e_g \cdot A \otimes_ST \cdot e_h = T \cdot (\pi_D^i\otimes 1),$$
where $0 \leqslant i < d$ is the unique integer such that $g = h \circ
\sigma^{-i}$. Hence, the map
$$\begin{xy}
(-17,0)*+{ e_g \cdot A \otimes_ST \cdot e_h }="1";
(17,0)*+{ e_{g \circ \sigma^{-1}} \cdot A \otimes_ST \cdot e_h }="2";
{ \ar^-{\varphi} "2";"1";};
\end{xy}$$
is an isomorphism, except for $g = h \circ \sigma$, where it
is injective with cokernel $k_T \cdot 1$.
\end{example}

Let $\mathcal{M}_{\,T,G}$ and $\mathcal{P}_{\,T,G}$ be the full
subcategories of $\mathsf{Mod}_{\,T,G}$ whose objects are the triples
$(N,(N_g),\varphi)$ such that the left $T$-module $N$ is finitely
generated and finitely generated and projective, respectively.

\begin{addendum}\label{add:gradedmodulesprojectives}The equivalence
$F \colon \mathsf{Mod}_{A \otimes_ST} \to \mathsf{Mod}_{\,T,G}$
restricts to equivalences
$F \colon \mathcal{M}_{A \otimes_ST} \to \mathcal{M}_{\,T,G}$ and
$F \colon \mathcal{P}_{A \otimes_ST} \to \mathcal{P}_{\,T,G}$.
\end{addendum}

\begin{proof}The $T$-algebra $A \otimes_ST$ is finitely generated and
projective as a $T$-module. So the former statement follows
immediately and to prove the latter, we must show that full
subcategory $\mathcal{P}_{\,T,G}$ of $\mathcal{M}_{\,T,G}$ is
precisely that consisting of the projective objects. Every projective
object in $\mathcal{M}_{\,T,G}$ is an object in $\mathcal{P}_{\,T,G}$,
since the underlying $T$-module of a projective left $A
\otimes_ST$-module is projective. To prove the 
converse, let $(P,(P_g),\varphi)$ be an object of
$\mathcal{P}_{\,T,G}$. The finitely generated projective
$T$-modules $P_g$ all have the same rank $r$. Indeed, the $T$-linear
map $\varphi^d \colon P_g \to P_g$ becomes an isomorphism
after extending scalars along $T \to L$, and hence, so does
$\varphi \colon P_g \to P_{g \circ \sigma^{-1}}$. Here and below we
use that $\sigma \in G$ is a generator. We call the common
rank $r$ the size of $(P,(P_g),\varphi)$ and proceed to show by
induction on $r$ that $(P,(P_g),\varphi)$ is a projective object in
$\mathcal{M}_{\,T,G}$, the case $r = 0$ being trivial.

First, if $r = 1$, then the $T$-modules $P_g/\pi_KP_g$ all have
length $1$. It follows that the maps
$\varphi \colon P_g \to P_{g \circ \sigma^{-1}}$ all are isomorphisms,
except for a single $g = h \circ \sigma$ for which it is injective
with cokernel of length $1$. We conclude from Example~\ref{ex:sizeone}
that $(P,(P_g),\varphi)$ is isomorphic to the object
$F(A \otimes_ST \cdot e_h)$, hence projective. 

To prove the induction step, we let $(P,(P_g),\varphi)$ have size $r >
1$ and assume that all objects of smaller size are projective. We will
construct a sequence in $\mathcal{P}_{\,T,G}$, 
$$\xymatrix{
{ 0 } \ar[r] &
{ (P',(P_g{\smash\!\!'}),\varphi') } \ar[r]^-{j} &
{ (P,(P_g),\varphi) } \ar[r]^-{q} &
{ (P'',(P_g''),\varphi'') } \ar[r] &
{ 0, } \cr
}$$
which is exact in the abelian category $\mathcal{M}_{\,T,G}$ and
in which the left-hand term has size $1$. Inductively, the left-hand
term and the right-hand term, which has size $r-1$, both are
projective, and hence, the sequence will show that also the middle term is
projective. To construct the desired sequence, we choose any 
non-zero element $x_1 \in P_{\,1}$ and set $x_g = \varphi^i(x_1) \in P_g$
if $g = \sigma^{-i}$ with $0 \leqslant i < d$. We then define
$j_g \colon P_g{\smash\!\!'} \to P_g$ by means of the pullback square
of $T$-modules
$$\xymatrix{
{ P_g{\smash\!\!'} } \ar[r]^-{j_g} \ar[d] &
{ P_g } \ar[d] \cr
{ L \cdot x_g } \ar[r] &
{ P_g \otimes_TL, } \cr
}$$
where all maps are the canonical inclusions, and define $q_g \colon
P_g \to P_g{\smash\!\!''}$ to be a cokernel of $j_g$. The $T$-modules
$P_g{\smash\!\!'}$ and $P_g{\smash\!\!''}$ are finitely generated
of rank $1$ and $r-1$, respectively, and both are torsion-free, and hence,
projective. Here, to see that $P_g{\smash\!\!''}$ is torsion-free, we
use that if $x \in P_g$ and $px \in P_g{\smash\!\!'}$, then $x \in
P_g{\smash\!\!'}$. We define $j \colon P' \to P$ and $q \colon P \to
P''$ to be the respective sums indexed by $g \in G$ of
$j_g \colon P_g{\smash\!\!'} \to P_g$ and $q_g \colon P_g \to
P_g{\smash\!\!''}$. The map $\varphi \colon P \to P$ induces maps
$\varphi' \colon P' \to P'$ and $\varphi'' \colon P'' \to
P''$. Moreover, since $\varphi$ is graded $T$-linear of degree
$\sigma^{-1}$ with $\varphi^d$ given by multiplication by $\pi_K \in
T$, the same is true for the maps $\varphi'$ and $\varphi''$. This
completes the proof.
\end{proof}

\begin{corollary}\label{cor:restriction-projective}Let $M$ be an
$A \otimes_ST$-module. If the $T \otimes_ST$-module
obtained from $M$ by restriction of scalars along
$T \otimes_ST \to A \otimes_ST$ is finitely generated and projective,
then $M$ is finitely generated and projective.
\end{corollary}

\begin{proof}Since restriction of scalars along $T \to T\otimes_ST$
takes $\mathcal{P}_{\,T\otimes_ST}$ to $\mathcal{P}_{\,T}$, it
suffices to prove that if the restriction of an $A \otimes_ST$-module
$M$ along $T \to A \otimes_ST$ is finitely generated projective, then
so is $M$. Applying $F$, this follows from the definition of
$\mathcal{P}_{\,T,G}$ and from
Addendum~\ref{add:gradedmodulesprojectives}.
\end{proof}

\begin{proof}[Proof of Theorem~\ref{thm:theoremC}]We claim that the
stronger statement that all submodules of a finitely generated
projective left $A \otimes_ST$-module again are finitely generated
projective holds. The analogous statement holds for the discrete
valuation ring $T$. Hence, the claim follows from
Addendum~\ref{add:gradedmodulesprojectives} and from $A \otimes_ST$ being
noetherian.
\end{proof}

We next identify the adjoint functors
$$\begin{xy}
(0,0)*+{ \mathcal{P}_{A \otimes_ST} }="1";
(25,0)*+{ \mathcal{P}_{\,T} }="2";
{ \ar@<.7ex>^-{\operatorname{Tr}_{A \otimes_ST/T}} "2";"1";};
{ \ar@<.7ex>^-{I_{A \otimes_ST/T}} "1";"2";};
\end{xy}$$
defined to be the restriction and coextension along
$T \to A \otimes_ST$ under the equivalence of
Addendum~\ref{add:gradedmodulesprojectives}. Given $g \in G$, we define
adjoint functors
$$\begin{xy}
(0,0)*+{ \mathcal{P}_{\,T,G} }="1";
(20,0)*+{ \mathcal{P}_{\,T} }="2";
{ \ar@<.7ex>^-{\smash{\operatorname{deg}}_{\hskip.5pt g}} "2";"1";};
{ \ar@<.7ex>^-{\operatorname{ind}_{\hskip.5pt g}} "1";"2";};
\end{xy}$$
by $\operatorname{deg}_g(P,(P_{\,h})_{h \in G},\varphi) = P_g$ and
$\operatorname{ind}_g(Q) = (\bigoplus_{h \in G}Q, (Q)_{h \in G},
\psi)$, where $\psi$ takes the summand indexed by $h \in G$ to the one
indexed by $\smash{ h \circ \sigma^{-1} \in G }$ by the map
$\pi_K \cdot \id_Q$, if $h = g$, and by the identity map,
otherwise. We define the adjunction isomorphism
$$\xymatrix{
{ \Hom_{T}(\operatorname{deg}_g(P,(P_{\,h})_{h \in G},\varphi), Q) }
\ar[r]^-{\alpha} &
{ \Hom_{A \otimes_ST}((P,(P_{\,h})_{h \in G},\varphi),
  \operatorname{ind}_g(Q)) } \cr
}$$
by $\alpha(f)_h = f \circ \varphi^{\,i} \colon P_{\, h} \to Q$, where
$g = h \circ \sigma^{-i}$ with $0 \leq i < d$. 

\begin{lemma}\label{lem:reducedtrace}In the situation above, the following hold:
\begin{enumerate}
\item[{\rm (i)}]The map induced by the canonical inclusions,
$$\xymatrix{
{ \bigoplus_{g \in G} (\deg_g \circ \,F)(P) } \ar[r] &
{ \Tr_{A \otimes_ST/T}(P), } \cr
}$$
is a natural isomorphism of left $T$-modules.
\item[{\rm (ii)}]Writing $g \in G$ as $g = \sigma^{\,-i}$ with
$0 \leq i < d$, the $T$-linear map
$$\begin{xy}
(0,0)*+{ (\deg_1 \circ \,F)(P) }="1";
(34,0)*+{ (\deg_{\hskip.5pt g} \circ\, F)(P) }="2";
{ \ar^-{(\pi_D \otimes 1)^i} "2";"1";};
\end{xy}$$
is natural and becomes an isomorphism after extension of scalars along
$T \to L$. 
\item[{\rm (iii)}]The multiplication
$\delta \colon T \otimes_ST \to T$ induces a natural isomorphism
$$\xymatrix{
{ (\operatorname{deg}_1 \circ\,F)(P) } \ar[r] &
{ \operatorname{Trd}_{A \otimes_ST/T}(P). } \cr
}$$
\end{enumerate}
\end{lemma}

\begin{proof}This follows immediately from the definitions.
\end{proof}

\begin{remark}\label{rem:reducedtrace}By adjunction, the
statements~(i)--(iii) in Lemma~\ref{lem:reducedtrace} imply that there
is a natural isomorphism of left $A \otimes_ST$-modules
$$\xymatrix{
{ I_{A \otimes_ST/T}(Q) } \ar[r] &
{ \prod_{g \in G} (H \circ \operatorname{ind}_g)(Q); } \cr
}$$
that there is a natural $A \otimes_ST$-linear map
$$\xymatrix{
{ (H \circ \operatorname{ind}_g)(Q) } \ar[r] &
{ (H \circ \operatorname{ind}_1)(Q), } \cr
}$$
which becomes an isomorphism after extension of scalars along $A
\otimes_ST \to D \otimes_KL$; and that there is a natural isomorphism
of left $A\otimes_ST$-modules
$$\xymatrix{
{ \Ird_{A \otimes_ST}(Q) } \ar[r] &
{ (H \circ \operatorname{ind}_1)(Q). } \cr
}$$
\end{remark}

Finally, we compare the $T$-order $A \otimes_ST$ in the semisimple
$L$-algebra $D \otimes_KL$ to a maximal $T$-order. We recall
from~\cite[Theorem~7.15]{reiner} that, by viewing $D$ as a right
$L$-vector space, left multiplication by $D$ on itself defines
an $L$-algebra isomorphism
\vspace{-1mm}
$$\begin{xy}
(0,0)*+{ D \otimes_KL }="1";
(25,0)*+{ \End_L(D). }="2";
{ \ar^-{l} "2";"1";};
\end{xy}$$
It restricts to a $T$-algebra monomorphism from the $T$-order
$A \otimes_ST$ of the domain to the $T$-order $\End_T(A)$ of the
target, which by op.~cit., Theorem~8.7, is a maximal
$T$-order. We identify the $T$-algebra $\End_T(A)$ with the matrix
$T$-algebra $M_d(T)$ by means of the $T$-algebra isomorphism
$\End_T(A) \to M_d(T)$ that to an endomorphism of the right $T$-module
$A$ associates its matrix with respect to the basis
$(\pi_D^s)_{0 \leqslant s < d}$.

\begin{proposition}\label{prop:milnorsquare}With notation as above,
there is a cartesian square of $T$-algebras
\vspace{-1mm}
$$\begin{xy}
(-14,7)*+{ A \otimes_ST }="11";
(14,7)*+{ B_d(k_T) }="12";
(-14,-7)*+{ M_d(T) }="21";
(14,-7)*+{ M_d(k_T) }="22";
{ \ar^-{i'} "12";"11";};
{ \ar^-{l} "21";"11";};
{ \ar^-(.46){\bar{l}} "22";"12";};
{ \ar^-{M_d(i)} "22";"21";};
\end{xy}$$
in which the right-hand vertical map is the canonical inclusion
of the subalgebra of lower triangular matrices, and the map $i \colon
T \to k_T$ is the canonical projection.
\end{proposition}

\begin{proof}The morphism $l$ maps the
$T$-subalgebra $T \otimes_ST \subset A \otimes_ST$ isomorphically
onto the $T$-subalgebra $H_d(T) \subset M_d(T)$ of diagonal
matrices, and
$$l(\pi_D \otimes 1) = 
\begin{pmatrix}
\, 0 \, & \, 0 \, & \, 0 \, & \cdots & \, 0 \, & \, \pi_K \, \cr
1 & 0 & 0 & \cdots & 0 & 0 \cr
0 & 1 & 0 & \cdots & 0 & 0 \cr
\vdots & \vdots & \vdots & \ddots & \vdots & \vdots \cr
0 & 0 & 0 & \cdots & 1 & 0 \cr
\end{pmatrix}.$$
Thus, the left $H_d(T)$-submodule of $M_d(T)$ spanned by 
$(l(\pi_D^s \otimes 1))_{0 \leqslant s < d}$ is equal to the left
$H_d(T)$-submodule of matrices that are lower triangular modulo
$M_d(\mathfrak{m}_T)$, and, by Corollary~\ref{lem:moretwisty}, this
left $H_d(T)$-submodule, in turn, is is equal to the image of the
$T$-algebra homomorphism $l \colon A \otimes_ST \to M_d(T)$.
\end{proof}

\begin{remark}\label{rem:milnorsquare}The top horizontal morphism
$i'$ in Proposition~\ref{prop:milnorsquare} maps the radical
$\mathfrak{m}_D \otimes_ST \subset A \otimes_ST$ onto the nilpotent
two-sided ideal $N_d(k_T) \subset B_d(k_T)$ of strictly lower triangular
matrices. Moreover, the composition
$$\xymatrix{
{ A \otimes_ST } \ar[r]^-{i'} &
{ B_d(k_T) } \ar[r] &
{ B_d(k_T)/N_d(k_T) } \ar[r] &
{ k_T \otimes_{k_S}k_T } \cr
}$$
of $i'$, the canonical projection, and the inverse of the isomorphism
that maps $t_1 \otimes t_2$ to the class of the diagonal matrix
$\operatorname{diag}(t_1t_2,\sigma^{-1}(t_1)t_2, \dots, \sigma^{-(d-1)}(t_1)t_2)$ is
equal to the canonical projection
$i \colon A \otimes_ST \to k_T \otimes_{k_S}k_T$. 
\end{remark}

\section{Localization sequences in $K$-theory and topological cyclic
  homology}\label{sec:localization}

In this section, we recall the algebraic $K$-theory symmetric spectrum
of a pointed exact category with weak equivalences following
Waldhausen~\cite[Section~1]{waldhausen} and prove a localization
sequence needed in the proof of Theorem~\ref{thm:theoremB}.

An exact category with weak equivalences is a triple
$(\mathcal{C},\mathcal{E},\mathcal{W})$ of a category $\mathcal{C}$;
a set $\mathcal{E}$ of exact sequences in $\mathcal{C}$ satisfying
axioms~(a) and~(b) in~\cite[\S2]{quillen}, axiom~(c) being
redundant~\cite[Appendix~A]{keller}; and a subcategory $\mathcal{W}$
of weak equivalences in $\mathcal{C}$ satisfying axioms~(Weq~1) and
(Weq~2) in~\cite[Section~1.2]{waldhausen}. An exact functor
$F \colon (\mathcal{C},\mathcal{E},\mathcal{W}) \to
(\mathcal{C}',\mathcal{E}',\mathcal{W}')$
between exact categories with weak equivalences is a functor
$F \colon \mathcal{C} \to \mathcal{C}'$ that maps $\mathcal{E}$ to
$\mathcal{E'}$ and $\mathcal{W}$ to $\mathcal{W}'$; and an exact
natural transformation between two such functors is a natural 
transformation $f \colon F \Rightarrow F'$ such that for every object
$c$ in $\mathcal{C}$, the morphism $f_c \colon F(c) \to F'(c)$ belongs
to $\mathcal{W}'$. Now Waldhausen's $S$-construction is a functor
$$\xymatrix{
{ \mathsf{ExCat} } \ar[r]^-{S} &
{ \mathsf{ExCat}^{\Delta^{\op}} } \cr
}$$
that to an exact category with weak equivalences assigns
a simplicial exact category with weak equivalences. The
construction, thus, may be iterated and gives, for every non-negative
integer $r$, a functor
$$\xymatrix{
{ \mathsf{ExCat} } \ar[r]^-{S^r} &
{ \mathsf{ExCat}^{(\Delta^{\op})^r} } \cr
}$$
that to an exact category with weak equivalences assigns an
$r$-simplicial exact category with weak equivalences. 

In the following, we will also write
$w(\mathcal{C},\mathcal{E},\mathcal{W})$ instead of $\mathcal{W}$ for 
the subcategory of weak equivalences. We define
$$(\mathcal{C}^w,\mathcal{E} \cap \mathscr{C}^w,\mathcal{W} \cap
\mathcal{C}^w) \subset (\mathcal{C},\mathcal{E},\mathcal{W})$$
be the full sub-exact category with weak equivalences consisting of
those objects $c$ in $\mathcal{C}$ with the property that $0 \to c$ is in
$\mathcal{W}$. Its subcategory of weak equivalences has a zero object,
and therefore, is contractible. In particular, the subspace
$$|N(w\,(\mathcal{C}^{w},\mathcal{E}\cap\mathcal{C}^{w}, 
\mathcal{W} \cap \mathcal{C}^{w}))| \; \subset \;
|N(w\,(\mathcal{C},\mathcal{E},\mathcal{W}))|$$
is contractible, and the pointed space given by the quotient
$$K(\mathcal{C},\mathcal{E},\mathcal{W})_r =
|N(w\,S^r(\mathcal{C},\mathcal{E},\mathcal{W}))|\,/\,
|N(w\,S^r(\mathcal{C}^{w},\mathcal{E}\cap\mathcal{C}^{w}, 
\mathcal{W} \cap \mathcal{C}^{w}))|$$
is by definition the $r$th space in the symmetric spectrum
$K(\mathcal{C},\mathcal{E},\mathcal{W})$. The left action by the
symmetric group $\Sigma_r$ on this pointed space is induced from the
permutation of the $r$-simplicial directions; and the spectrum
structure maps
$$\begin{xy}
(0,0)*+{ K(\mathcal{C},\mathcal{E},\mathcal{W})_r \wedge S^s }="1";
(37,0)*+{ K(\mathcal{C},\mathcal{E},\mathcal{W})_{r+s} }="2";
{ \ar^-{\sigma_{r,s}} "2";"1";};
\end{xy}$$
are induced by the inclusion of the $1$-skeleta in the last $s$
simplicial directions. We refer to~\cite[Appendix]{gh} for proof
that this is indeed a symmetric spectrum and for a discussion of
multiplicative properties of the construction. We also recall that, as
a consequence of the additivity theorem~\cite[Theorem~1.4.2,
Proposition~1.5.3]{waldhausen}, the symmetric spectrum 
$K(\mathcal{C},\mathcal{E},\mathcal{W})$ is fibrant in the positive
model structure on the category of symmetric spectra;
see~\cite[Section~14]{mandellmayshipleyschwede}.

An exact functor 
$F \colon (\mathcal{C},\mathcal{E},\mathcal{W}) \to
(\mathcal{C}',\mathcal{E}',\mathcal{W}')$ induces a morphism
$$\begin{xy}
(0,0)*+{ K(\mathcal{C},\mathcal{E},\mathcal{W}) }="1";
(33,0)*+{ K(\mathcal{C}',\mathcal{E}',\mathcal{W}'), }="2";
{ \ar^-{K(F)} "2";"1";};
\end{xy}$$
of symmetric spectra and an exact natural transformations
$f \colon F \Rightarrow F'$ between two such functors gives rise to a
homotopy $K(f)$ from $K(F)$ to $K(F')$. In this way, the $K$-theory
construction is a strict 2-functor from the strict 2-category of exact
categories with weak equivalences, exact functors, and exact natural
transformations to the strict 2-category of symmetric spectra,
morphisms of symmetric spectra, and homotopy classes of homotopies
between morphisms of symmetric spectra. Like every 2-functor, it takes
adjunctions in the domain 2-category to adjunctions in the target
2-category, and the latter adjunctions automatically are adjoint
equivalences, since the 2-morphisms in the target 2-category are
invertible.

An abelian category $\mathcal{M}$ has a canonical structure of an
exact category with weak equivalences, where the set of exact
sequences $\mathcal{E}$ consists of the sequences
$$\xymatrix{
{ M' } \ar[r]^-{i} &
{ M } \ar[r]^-{p} &
{ M'' } \cr
}$$
in $\mathcal{M}$ such that $i$ is a kernel of $p$ and $p$ a cokernel
of $i$, and where the subcategory of weak equivalences $\mathcal{W}$
is the subcategory of isomorphisms in $\mathcal{M}$. Moreover, an
additive full subcategory $\mathcal{P}$ of $\mathcal{M}$, which is
extension-closed in the sense that, for every sequence in
$\mathcal{E}$ whose initial term $M'$ and terminal term $M''$ are in
$\mathcal{P}$, also the middle term $M$ is in $\mathcal{P}$, has an
induced structure of exact category with weak equivalences, where the
set of exact sequences $\mathcal{E} \cap \mathcal{P}$ consists of the
sequences in $\mathcal{E}$ all of whose terms are in $\mathcal{P}$,
and where the subcategory of weak equivalences $\mathcal{W} \cap
\mathcal{P}$ is the full subcategory of $\mathcal{W}$ whose objects
are in $\mathcal{P}$.

Now let $R$ be a left noetherian ring, and let $\mathcal{M}_R$ and
$\mathcal{P}_R$ be the categories of finitely generated left
$R$-modules and finitely generated and projective left $R$-modules,
respectively. We assume that these categories are small, which may be
accomplished by assuming the axiom of
universe~\cite[Expos\'{e}]{SGA4I} or by some ad hoc restriction on the
modules allowed. The category $\mathcal{M}_R$ is abelian and the
additive full subcategory $\mathcal{P}_R$ is extension-closed. We
define $K'(R)$ and $K(R)$ to be the $K$-theory symmetric spectra
of the exact categories with weak equivalences associated to these as
discussed above. We recall that the canonical inclusion functor
induces a weak equivalence
$$\xymatrix{
{ K(R) } \ar[r] &
{ K'(R), } \cr
}$$
provided that $R$ is left regular in the sense that every object in
$\mathcal{M}_R$ admits a finite resolution by objects in
$\mathcal{P}_R$.

\begin{proposition}\label{prop:localizationsequence}Let $S$ be a
complete discrete valuation ring with residue field $k_S$ and quotient
field $K$ and let $R$ be an $S$-algebra. Assuming that, as a ring, $R$
is left regular, there is a canonical natural cofibration sequence of
symmetric spectra
$$\xymatrix{
{ K'(R \otimes_Sk_S) } \ar[r]^-{i_*} &
{ K(R) } \ar[r]^-{j^*} &
{ K(R \otimes_SK) } \ar[r]^-{\partial} &
{ \Sigma K'(R \otimes_Sk_S). }
}$$
The terms in the sequence have canonical natural $K(S)$-module
structures and the maps in the sequences respect these structures.
\end{proposition}

\begin{proof}We first introduce some notation. Let $\mathcal{M}$
be an abelian category, let $\mathcal{P}$ be an extension-closed full
additive subcategory of $\mathcal{M}$, and let $\mathcal{H}$ and
$\mathcal{T}$ be two Serre subcategories of $\mathcal{M}$. Let
$\mathcal{E} \cap \mathcal{P}$ be the set of exact sequences in the
exact category structure on $\mathcal{P}$ defined above. We define
$\mathsf{Ch}^b(\mathcal{P},\mathcal{H},\mathcal{T})$ to be the
following exact category with weak equivalences: The
underlying category has objects the bounded chain complexes in
$\mathcal{P}$ whose associated homology objects, calculated in
$\mathcal{M}$, are in the Serre subcategory $\mathcal{H}$, and has
morphisms all chain maps; the exact sequences in this category are the
sequences of complexes that degree-wise are in
$\mathcal{E} \cap \mathcal{P}$; and the weak equivalences are the
morphisms that, modulo the Serre subcategory $\mathcal{T}$, induce
isomorphisms of homology objects. If $\mathcal{T}$ is the Serre
subcategory of zero objects in $\mathcal{M}$, then we write
$\mathsf{Ch}^b(\mathcal{P},\mathcal{H})$ instead of
$\mathsf{Ch}^b(\mathcal{P},\mathcal{H},\mathcal{T})$.

We now let $\mathcal{T}_R$ be the Serre subcategory of $\mathcal{M}_R$
whose objects are the finitely generated left $R$-modules annihilated
by extension of scalars along $f \colon S \to K$ and consider the
diagram of $K$-theory symmetric spectra
$$\begin{xy}
(-20,7)*+{ K(\mathsf{Ch}^b(\mathcal{P}_R,\mathcal{T}_R)) }="11";
(20,7)*+{ K(\mathsf{Ch}^b(\mathcal{P}_R,\mathcal{T}_R,\mathcal{T}_R))
}="12";
(-20,-7)*+{ K(\mathsf{Ch}^b(\mathcal{P}_R,\mathcal{M}_R)) }="21";
(20,-7)*+{ K(\mathsf{Ch}^b(\mathcal{P}_R,\mathcal{M}_R,\mathcal{T}_R))
}="22";
{ \ar^-{J'} "12";"11";};
{ \ar^-{I} "21";"11";};
{ \ar^-(.45){I'} "22";"12";};
{ \ar^-{J} "22";"21";};
\end{xy}$$
with all maps induced by the respective canonical inclusion
functors. It follows from Waldhausen's fibration
theorem~\cite[Theorem~1.6.4]{waldhausen} that the diagram is homotopy
cartesian. Moreover, by the 2-functoriality of the $K$-theory
construction, the unique natural transformation from the identity
functor to a constant functor with value a zero object defines a
homotopy from the identity map of the upper right-hand term to the
constant map. Indeed, this natural transformation is exact. This
homotopy, in turn, determines a homotopy from the composite map $J
\circ I$ to the constant map, and the combined data determines a
cofibration sequence in the homotopy category of symmetric spectra. We
proceed to identify the terms and maps in this cofibration sequence
with the ones in the cofibration sequence in the statement.

We use Waldhausen's approximation
theorem~\cite[Theorem~1.6.7]{waldhausen}, but for our purposes, the
formulation in~\cite[Theorem~1.9.8]{thomasontrobaugh} is more
convenient. The theorem states that the map of $K$-theory symmetric
spectra induced by an exact functor $F$ is a weak equivalence, if a
list of hypotheses are satisfied. In our situation, the only
hypothesis that is not automatically satisfied is loc.~cit.~1.9.7.1,
which is the requirement that, for every object $B$ in the target of
$F$, there exists an object $A$ in the domain of $F$ and a weak
equivalence $f \colon F(A) \to B$ in the target of $F$. 

Now, the left-hand 
vertical map $I$ fits in the diagram 
$$\xymatrix{
{ K(\mathsf{Ch}^b(\mathcal{P}_R,\mathcal{T}_R)) } \ar[r]^-{I} \ar[d] &
{ K(\mathsf{Ch}^b(\mathcal{P}_R,\mathcal{M}_R)) } \ar@{=}[r] \ar[d] &
{ K(\mathsf{Ch}^b(\mathcal{P}_R,\mathcal{M}_R)) } \ar@{=}[d]  \cr
{ K(\mathsf{Ch}^b(\mathcal{M}_R,\mathcal{T}_R)) } \ar[r] &
{ K(\mathsf{Ch}^b(\mathcal{M}_R,\mathcal{M}_R)) } &
{ K(\mathsf{Ch}^b(\mathcal{P}_R,\mathcal{M}_R)) } \ar[l] \ar@{=}[d] \cr
{ K(\mathsf{Ch}^b(\mathcal{T}_R,\mathcal{T}_R)) } \ar[r] \ar[u] &
{ K(\mathsf{Ch}^b(\mathcal{M}_R,\mathcal{M}_R)) } \ar@{=}[u] &
{ K(\mathsf{Ch}^b(\mathcal{P}_R,\mathcal{M}_R)) } \ar[l] \cr
{ K(\mathcal{T}_R) } \ar[r] \ar[u]_-{F_0} &
{ K(\mathcal{M}_R) } \ar[u]_-{F_0} &
{ K(\mathcal{P}_R) } \ar[u]_-{F_0} \ar[l] \cr
{ K(\mathcal{M}_{R \otimes_Sk_S}) } \ar[u]_-{i_*} \ar[r]^-{i_*} &
{ K(\mathcal{M}_R) } \ar@{=}[u] &
{ K(\mathcal{P}_R), } \ar@{=}[u] \ar[l] \cr
}$$
where the unmarked maps are induced by the respective canonical
inclusion functors. The top vertical maps and the second and third
right-hand horizontal maps are weak equivalences by
op.~cit.~Theorem~1.9.8, and the vertical maps labelled $F_0$ are weak
equivalences by op.~cit.~Theorem~1.11.7. The second left-hand vertical
map also is a weak equivalence by op.~cit.~Theorem 1.9.8, the
hypothesis~1.9.7.1 being satisfied
by~\cite[Lemma~1.5.3]{hm4}. Finally, the lower left-hand vertical map
is a weak equivalence by Quillen's devissage
theorem~\cite[Theorem~4]{quillen}.

Similarly, the lower horizontal map $J$ in the diagram above fits in the
diagram
$$\xymatrix{
{ K(\mathsf{Ch}^b(\mathcal{P}_R,\mathcal{M}_R)) } \ar[r]^-{J} &
{ K(\mathsf{Ch}^b(\mathcal{P}_R,\mathcal{M}_R,\mathcal{T}_R)) } \ar[r]^-{j^*}
&
{ K(\mathsf{Ch}^b(\mathcal{P}_{R \otimes_SK},\mathcal{M}_{R
    \otimes_SK})) } \cr
{ K(\mathcal{P}_R) } \ar[rr]^-{j^*} \ar[u]_{F_0} &
{} &
{ K(\mathcal{P}_{R \otimes_SK}) } \ar[u]_{F_0} \cr
}$$
with the vertical maps induced by the exact functor that to a module
$M$ assigns the complex $F_0(M)$ whose degree $n$ term is $M$, if
$n = 0$, and a zero object $0$, otherwise. The upper right-hand
horizontal map is induced by the exact functor given by degree-wise
extension of scalars along $f \colon S \to K$, and it is a weak
equivalence, since this functor satisfies the hypotheses
op.~cit.~Theorem~1.9.8. The vertical maps also are weak equivalences
by op.~cit.~Theorem~1.11.7. 
\end{proof}

We will prove an analogue of
Proposition~\ref{prop:localizationsequence} for topological cyclic
homology, and begin by recalling the definition
following~\cite{dundasmccarthy}. We denote by $\mathbb{T}$ 
the circle group of complex numbers of modulus $1$ under
multiplication.

For $(\mathcal{C},\mathcal{E},\mathcal{W})$ an exact
category with weak equivalences, the B\"{o}kstedt-Dennis
trace map is a natural morphism of symmetric spectra with left
$\mathbb{T}$-action
\vspace{-1mm}
$$\begin{xy}
(0,0)*+{ K(\mathcal{C},\mathcal{E},\mathcal{W}) }="1";
(32,0)*+{ \THH(\mathcal{C},\mathcal{E},\mathcal{W}) }="2";
{ \ar^-{\tr} "2";"1";};
\end{xy}$$
from the $K$-theory symmetric spectrum with trivial left
$\mathbb{T}$-action to the topological Hochschild spectrum, the
definition and properties of which we now briefly discuss. The
topological Hochschild construction assigns to an additive
category $\mathcal{C}$ the left $\mathbb{T}$-space
$\THH(\mathcal{C})$ defined to be the realization of the
cyclic space $\THH(\mathcal{C})[-]$ given
in~\cite[Definition~1.3.6]{dundasmccarthy}. The left
$\mathbb{T}$-action is a consequence of Connes' theory of cyclic 
objects~\cite{drinfeld}, which also identifies the canonical inclusion of
the subspace of points fixed by the left $\mathbb{T}$-action with a
$\mathbb{T}$-equivariant map
\vspace{-1mm}
$$\begin{xy}
(0,0)*+{ \ob(\mathcal{C}) }="1";
(23,0)*+{ \THH(\mathcal{C}) }="2";
{ \ar^-{\tr} "2";"1";};
\end{xy}$$
from the set of objects in $\mathcal{C}$ considered as a discrete
space with trivial left $\mathbb{T}$-action. The topological
Hochschild symmetric spectrum of an exact category with weak
equivalences $(\mathcal{C},\mathcal{E},\mathcal{W})$ and the 
B\"{o}kstedt-Dennis trace map is defined by incorporating
Waldhausen's $S$-construction as follows. If $I$ is a small category,
then we define $(\mathcal{C},\mathcal{E},\mathcal{W})^I$ to be the
exact category with the category $\mathcal{C}^I$ of $I$-diagrams in 
$\mathcal{C}$ as underlying category; with the sequences in
$\mathcal{C}^I$ that, objectwise, are in $\mathcal{E}$ as the exact
sequences; and with the morphisms in $\mathcal{C}^I$ that, objectwise,
are in $\mathcal{W}$ as the weak equivalences. We also fix the
fully faithful functor
$$i \colon \Delta \to \mathsf{Cat}$$
that to a non-empty finite ordinal $[n]$
assigns the category $i([n])$ with object set $[n]$ and with a unique
morphism from $s$ to $t$ if and only if $s \leqslant t$, and that to
an order-preserving map $\theta \colon [m] \to [n]$ assigns the unique
functor $i(\theta) \colon i([m]) \to i([n])$ with the map $\theta$ as
the underlying map of object sets. With these preparations in hand, we
let
$$N^w(\mathcal{C},\mathcal{E},\mathcal{W})[-] \,\subset\,
(\mathcal{C},\mathcal{E},\mathcal{W})^{i([-])}$$
be the sub-simplicial exact category with weak equivalences, whose
underlying simplicial category is the \emph{full} sub-simplicial
category of $\mathcal{C}^{i([-])}$ with simplicial set of objects
given by $\ob(\mathcal{W}^{i([-])})$. Again, the subspace
$$| \THH(N^w(S^r(\mathcal{C}^w, \mathcal{E} \cap \mathcal{C}^w,
\mathcal{W} \cap \mathcal{C}^w))) | \; \subset \; |
\THH(N^w(S^r(\mathcal{C}, \mathcal{E}, \mathcal{W}))) |$$
is $\mathbb{T}$-equivariantly contractible, and the $r$th space in the
symmetric spectrum with left $\mathbb{T}$-action
$\THH(\mathcal{C},\mathcal{E},\mathcal{W})$ is defined to be the
quotient pointed left $\mathbb{T}$-space. Here, the topological
Hochschild construction is applied degreewise to the 
$(r+1)$-simplicial additive categories in question. The structure maps
\vspace{-1mm}
$$\begin{xy}
(0,0)*+{ \THH(\mathcal{C},\mathcal{E},\mathcal{W})_r \wedge S^s }="1";
(40,0)*+{ \THH(\mathcal{C},\mathcal{E},\mathcal{W})_{r+s} }="2";
{ \ar^-{\sigma_{r,s}} "2";"1"; };
\end{xy}$$
are defined analogously to those in the $K$-theory symmetric
spectrum.

We next recall the classical definition of the Frobenius maps
$$\xymatrix{
{ \THH(\mathcal{C},\mathcal{E},\mathcal{W}) } \ar[r]^-{\varphi_p} &
{ \THH(\mathcal{C},\mathcal{E},\mathcal{W})^{tC_p} } \cr
}$$
following~\cite[Section~III.5]{nikolausscholze}.\footnote{\,Recently,
a fully homotopy invariant definition based on the Tate diagonal
of~\cite[Section~III.1]{nikolausscholze} was given by
Nikolaus~\cite{nikolaus}.} We define maps of symmetric spectra with
left $\mathbb{T}$-action
$$\xymatrix{
{ \THH(\mathcal{C},\mathcal{E},\mathcal{W}) } &
{ \rho_p^*\THH(\mathcal{C},\mathcal{E},\mathcal{W})^{\phi C_p} } 
\ar[l]_-(.48){r_p} \ar[r]^-{s_p} &
{ \rho_p^*\THH(\mathcal{C},\mathcal{E},\mathcal{W})^{tC_p}, } \cr
}$$
where $C_p \subset \mathbb{T}$ is the subgroup of prime order $p$, and
where $\mathbb{T}$ acts on the middle and right-hand terms via the
$p$th root isomorphism
$\rho_p \colon \mathbb{T} \to \mathbb{T}/C_p$. The left-hand map will
be an equivalence, and hence, we may compose an inverse of $r_p$ with the
map $s_p$ to obtain the Frobenius map $\varphi_p$,
well-defined up to contractible choice, in the infinity-category
$\mathsf{Sp}^{B\mathbb{T}}$ of spectra with left
$\mathbb{T}$-action. The topological Hochschild 
construction~\cite[Section~1.3.6]{dundasmccarthy} gives, more
generally, a functor $\THH(\mathcal{C},\mathcal{E},\mathcal{W};-)$
from pointed spaces to symmetric spectra with left
$\mathbb{T}$-action, whose value at $S^0$ agrees with
$\THH(\mathcal{C},\mathcal{E},\mathcal{W})$. In order to define the
geometric fixed point spectrum and the Tate spectrum, we fix the
infinite dimensional complex $\mathbb{T}$-representation
$$\mathcal{U} = \bigoplus_{k \in \Z, i \in \N} \C_{k,i},$$
where $z \in \mathbb{T}$ acts on $\C_{k,i}$ by multiplication by
$z^{\,k}$. We will write $V \subset \mathcal{U}$ to
indicate that $V$ is a finite dimensional complex
sub-$\mathbb{T}$-representation of $\mathcal{U}$. Now, the pointed
space $\THH(\mathcal{C},\mathcal{E},\mathcal{W};S^V)_r$ has
two left $\mathbb{T}$-actions, one coming from the cyclic structure
and one induced by the left $\mathbb{T}$-action on $S^V$. If
we give it the diagonal $\mathbb{T}$-action, then
$$(\THH(\mathcal{C},\mathcal{E},\mathcal{W})^{\phi C_p})_r =
\operatornamewithlimits{hocolim}_{V \subset \mathcal{U},
  V^{C_p} = 0}
(\THH(\mathcal{C},\mathcal{E},\mathcal{W};S^V)_r^{C_p})$$
is the $r$th space of the geometric fixed point spectrum. Similarly,
$$(\THH(\mathcal{C},\mathcal{E},\mathcal{W})^{tC_p})_r =
\operatornamewithlimits{hocolim}_{V \subset \mathcal{U},
  V^{C_p} = 0}
(\THH(\mathcal{C},\mathcal{E},\mathcal{W};S^V)_r^{hC_p})$$
is the $r$th space of the Tate spectrum, and the map $s_p$ is induced
by the canonical map from fixed points to homotopy fixed points. The
map $r_p$, in turn, is the map from the first homotopy colimit induced
by $\mathbb{T}$-equivariant maps
$$\xymatrix{
{
  \rho_p^*(\THH(\mathcal{C},\mathcal{E},\mathcal{W};S^V)_r^{C_p})
} \ar[r] &
{ \THH(\mathcal{C},\mathcal{E},\mathcal{W};S^{\,0})_r } \cr
}$$
that exist in the B\"{o}kstedt model of topological Hochschild
homology and are given by restricting a $C_p$-equivariant map to the
induced map of $C_p$-fixed points. It is an equivalence
by~\cite[Theorem~III.4.7]{nikolausscholze}. 

In general, one uses the Frobenius maps to define a number of spectra
associated with a cyclotomic spectrum $X$. If $p$ is a prime number
and $s \geqslant 1$ an integer, then the Tate orbit
lemma,~\cite[Lemma~I.2.1]{nikolausscholze}, implies that the
canonical map
$$\xymatrix{
{ X^{tC_{p^s}} } \ar[r] &
{ (X^{tC_p})^{h(C_{\smash{p^s}}/C_p)}} \cr
}$$
is an equivalence.
%
%
%%%%% The spectrum $X$ must be bounded below for this to be true.
%
%
Hence, by precomposing the inverse equivalence with
the map of homotopy $C_{\smash{p^{s-1}}}$-fixed points induced by
$\varphi_p \colon X \to X^{tC_p}$, we obtain a map
$$\xymatrix{
{ X^{hC_{\smash{p^{s-1}}}} } \ar[r] &
{ X^{tC_{\smash{p^s}}} } \cr
}$$
and define
$$\TR^n(X;p) = X \times_{\textstyle{X^{tC_p}}} X^{hC_p} 
\times_{\textstyle{X^{tC_{\smash{p^2}}}}}
X^{hC_{\smash{p^2}}} \times_{\textstyle{X^{tC_{\smash{p^3}}}}} \cdots
\times_{\textstyle{X^{tC_{\smash{p^{n-1}}}}}} X^{hC_{\smash{p^{n-1}}}},$$
where $X^{hC_{\smash{p^{s-1}}}} \to X^{tC_{\smash{p^s}}}$ are the maps
just defined, whereas
$X^{tC_{\smash{p^s}}} \leftarrow X^{hC_{\smash{p^s}}}$ are the
canonical maps. Moreover, the restriction and Frobenius maps
$$\xymatrix{
{ \TR^n(X;p) } \ar@<.7ex>[r]^-{R} \ar@<-.7ex>[r]_-{F} &
{ \TR^{n-1}(X;p) } \cr
}$$
are defined to be the projection onto the first $n-1$
factors and the composition of the projection onto the last $n-1$
factors and the map induced by the forgetful map $X^{hC_p} \to X$,
respectively. Since the fiber of the restriction map agrees with the
fiber of the canonical map 
$X^{hC_{\smash{p^{n-1}}}} \to X^{tC_{\smash{p^{n-1}}}}$, we get the
``fundamental'' cofibration sequence 
$$\xymatrix{
{ X_{hC_{\smash{p^{n-1}}}} } \ar[r]^-{N} &
{ \TR^n(X;p) } \ar[r]^-{R} &
{ \TR^{n-1}(X;p). } \cr
}$$
Finally, we define $\TC^n(X;p)$ to be the homotopy equalizer of
these two maps and define $\TC(X;p)$ to be their homotopy 
limit as $n \geqslant 1$ varies. It is proved
in~\cite[Theorem~II.4.10]{nikolausscholze} that if $X$ is
$p$-complete, then the spectrum $\TC(X;p)$ agrees canonically with the
spectrum $\TC(X)$ considered in the introduction. 

If $X$ is the cyclotomic spectrum
$\THH(\mathcal{C},\mathcal{E},\mathcal{W})$, then the 
B\"{o}kstedt-Dennis trace map lifts to a map of spectra called the
cyclotomic trace map
$$\xymatrix{
{ K(\mathcal{C},\mathcal{E},\mathcal{W}) } \ar[r]^-{\tr} &
{ \TC(\mathcal{C},\mathcal{E},\mathcal{W}). } \cr
}$$
It is not clear from the definition that topological cyclic homology
should be easier to understand than $K$-theory. We now explain
why this is often the case.

As for $K$-theory, the additivity
theorem~\cite[Proposition~2.0.4]{dundasmccarthy} implies that the
adjunct structure maps
\vspace{-1mm}
$$\xymatrix{
{ \THH(\mathcal{C},\mathcal{E},\mathcal{W})_r }
\ar[r]^-{\tilde{\sigma}_{r,s}} &
{ \Omega^s(\THH(\mathcal{C},\mathcal{E},\mathcal{W})_{r+s}) } \cr
}$$
are weak equivalences, for all integers $r \geqslant 1$ and
$s \geqslant 0$. However, by contrast with $K$-theory, these maps are
also weak equivalences, for $r = 0$ and $s \geqslant 0$, if $\mathcal{E}$ is 
equal to the set of split-exact sequences in $\mathcal{C}$;
see~\cite[Proposition~2.1.3]{dundasmccarthy}. Moreover, the
inclusion of the $0$-skeleton in
$N^w(\mathcal{C},\mathcal{E},\mathcal{W})$ induces a weak equivalence
$$\xymatrix{
{ \THH(\mathcal{C}) } \ar[r] &
{ \THH(\mathcal{C},\mathcal{E},\mathcal{W})_0, } \cr
}$$
if $\mathcal{W}$ is equal to the subcategory of isomorphisms in
$\mathcal{C}$.

Now let $A$ be a unital associative ring. By abuse of notation, we
write $A$ also for the additive category with a single object
$\emptyset$ whose ring of endomorphisms is $A$. There is an additive
functor $i \colon A \to \mathcal{P}_A$ that to the unique object
$\emptyset$ assigns $A$ considered as a left $A$-module under
multiplication and that to $a \in \End_A(\emptyset)$ assigns the
$A$-linear map $i(a) \in \End_{\mathcal{P}_A}(i(\emptyset))$ given by
right multiplication by $a$. The induced map
$$\begin{xy}
(0,0)*+{ \THH(A) }="1";
(28,0)*+{ \THH(\mathcal{P}_A) }="2";
{ \ar^-{\THH(i)} "2";"1";};
\end{xy}$$
is a weak equivalence by~\cite[Proposition~2.1.5]{dundasmccarthy}. The
domain of this map is B\"{o}kstedt's original topological Hochschild
space of the ring $A$, whose homotopy groups often are amenable
to calculation. Hence, we conclude that B\"{o}kstedt's topological
Hochschild homology groups and the homotopy groups of the topological
Hochschild spectrum of the category $\mathcal{P}_A$ equipped with its
canonical structure of exact category with weak equivalences agree, up
to canonical natural isomorphism. We will abuse notation and write
$\THH(A)$ also for the latter spectrum, and we write $\TR^n(A;p)$,
$\TC^n(A;p)$, etc.~for the associated spectra defined above. The
homotopy groups $\THH_q(A)$, however, are well-defined, up to
canonical natural isomorphism. 

We are now in a position to state and prove the analogue of
Proposition~\ref{prop:localizationsequence} for topological Hochschild
homology and its variants.

\begin{proposition}\label{prop:thhlocalizationsequence}Let $S$ be a
complete discrete valuation ring with residue field $k_S$ and quotient
field $K$ and let $R$ be an $S$-algebra. Assuming that, as a ring, $R$
is left regular, there is a canonical natural cofibration sequence of
cyclotomic spectra
$$\xymatrix{
{ \THH'(R \otimes_Sk_S) } \ar[r]^-{i_*} &
{ \THH(R) } \ar[r]^-{j^*} &
{ \THH(R \, | \, R \otimes_SK) } \ar[r]^-{\partial} &
{ \Sigma\THH'(R \otimes_Sk_S), } \cr
}$$
where the left-hand term and right-hand term denote the topological
Hochschild spectra of $\mathcal{M}_{R \otimes_Sk_S}$ and
$\mathsf{Ch}^b(\mathcal{P}_R,\mathcal{M}_R,\mathcal{T}_R)$,
respectively. The terms in the sequence have canonical
$\,\THH(S)$-module structures and the morphisms in the sequence are
$\,\THH(S)$-linear.
\end{proposition}

\begin{proof}We repeat the proof of
Proposition~\ref{prop:localizationsequence} mutatis mutandis. The analogues
of Waldhausen's fibration theorem and Quillen's d\'{e}vissage theorem
hold and are proved in~\cite[Theorem~1.3.11]{hm4}
and~\cite[Theorem~1]{dundas}, respectively. However, the analogue
of Waldhausen's approximation theorem, which is proved
in~\cite[Proposition~2.3.2]{dundasmccarthy}, requires a
stronger hypothesis to be satisfied. In the situation at hand, it
follows from~\cite[Lemma~1.5.3]{hm4} and~\cite[Chapter~XVII,
Proposition~1.2]{cartaneilenberg} that the hypothesis is satisfied in
all cases, with the exception that $\THH(R\,|\,R \otimes_SK)$
cannot be identified with $\THH(R \otimes_SK)$. 
\end{proof}

\begin{addendum}\label{add:localizationsequence}Let $S$ be a complete
discrete valuation ring with residue field $k_S$ and quotient field
$K$ and let $R$ be an $S$-algebra. Assuming that, as a ring, $R$ is
left regular, there is a commutative diagram of spectra
$$\xymatrix{
{ K'(R \otimes_Sk_S) } \ar[r]^-{i_*} \ar[d]^-{\tr} &
{ K(R) } \ar[r]^-{j^*} \ar[d]^-{\tr} &
{ K(R \, | \, R \otimes_SK) } \ar[r]^-{\partial} \ar[d]^-{\tr} &
{ \Sigma K'(R \otimes_Sk_S) } \ar[d]^-{\Sigma\tr} \cr
{ \TC'(R \otimes_Sk_S) } \ar[r]^-{i_*} &
{ \TC(R) } \ar[r]^-{j^*} &
{ \TC(R \, | \, R \otimes_SK) } \ar[r]^-{\partial} &
{ \Sigma\TC'(R \otimes_Sk_S) } \cr
}$$
in which the rows are cofibration sequences.
\end{addendum}

\begin{proof}The two sequences in the statement are obtained by
applying respectively the $K$-theory functor and the topological
cyclic homology functor to the same sequence of pointed exact
categories with weak equivalences, so the diagram commutes by the
naturality of the cyclotomic trace map. The sequences are cofibration
sequences by Proposition~\ref{prop:localizationsequence}
and Proposition~\ref{prop:thhlocalizationsequence}, respectively.
\end{proof}

\begin{remark}\label{rem:artinian}If the ring $R \otimes_Sk_S$ is artinian
and if its quotient $(R \otimes_S k_S)/J$ by the radical is regular, then
the morphisms
$$\xymatrix{
{ K((R \otimes_Sk_S)/J) } \ar[r] &
{ K'((R \otimes_Sk_S)/J) } \ar[r] &
{ K'(R \otimes_Sk_S) } \cr
}$$
induced by the canonical inclusion functor and by the
restriction-of-scalars functor are weak equivalences
by~\cite[Corollary~2 of Theorem~3]{quillen}
and~\cite[Theorem~4]{quillen}. The analogous statements for
topological Hochschild homology and its variants hold by~\cite[Theorem~2]{dundas}
and~\cite[Theorem~1]{dundas}.
\end{remark}

\begin{theorem}\label{thm:KofD}With notation as in the introduction, the
cyclotomic trace map induces isomorphisms of $p$-adic homotopy groups
in degrees $j \geqslant 1$,
\vspace{-1mm}
$$\begin{xy}
(0,0)*+{ K_j(D,\Zp) }="1";
(30,0)*+{ \TC_j(A \,|\, D,\Zp). }="2";
{ \ar^-{\tr} "2";"1";};
\end{xy}$$
\end{theorem}

\begin{proof}The ring $A \otimes_Sk_S$ is artinian and its quotient
$k_T$ by the radical is regular. Hence, by Remark~\ref{rem:artinian},
Addendum~\ref{add:localizationsequence} gives a commutative diagram of
symmetric spectra
$$\xymatrix{
{ K(k_T) } \ar[r]^-{i_*} \ar[d]^-{\tr} &
{ K(A) } \ar[r]^-{j^*} \ar[d]^-{\tr} &
{ K(A \, | \, D) } \ar[r]^-{\partial} \ar[d]^-{\tr} &
{ \Sigma K(k_T) } \ar[d]^-{\Sigma\tr} \cr
{ \TC(k_T;p) } \ar[r]^-{i_*} &
{ \TC(A;p) } \ar[r]^-{j^*} &
{ \TC(A \, | \, D;p) } \ar[r]^-{\partial} &
{ \Sigma\TC(k_T;p) } \cr
}$$
in which the rows are cofibration
sequences. By~\cite[Theorem~D]{hm}, the first and second vertical
morphisms from the left induce isomorphisms of $p$-adic homotopy
groups in non-negative degrees, so the theorem follows from the
five-lemma.
\end{proof}

\section{Galois descent for topological cyclic
homology}\label{sec:etale}

In this section, we prove a rather general \'{e}tale descent result
for topological cyclic homology. The following result is well-known
for commutative algebras.

\begin{proposition}\label{prop:quasicoherent}Let $S$ be a commutative
ring, let $f \colon S \to T$ be an \'{e}tale morphism of commutative
rings, and let $A$ be any unital associative $S$-algebra, not
necessarily commutative. In this case, the map induced by extension of
scalars along $f \colon S \to T$,
$$\xymatrix{
{ \THH_j(A) \otimes_ST } \ar[r] &
{ \THH_j(A \otimes_ST), } \cr
}$$
is an isomorphism, for all integers $j$.
\end{proposition}

\begin{proof}To fix notation, given a commutative ring $S$, a unital
associative $S$-algebra $R$, and an $S$-symmetric $R$-$R$-bimodule
$M$, one has the Hochschild homology $S$-modules $\HH_j(R/S;M)$.
We abbreviate and write $\HH_j(R/S)$ instead of $\HH_j(R/S;M)$, if $M$
is $R$ considered as an $S$-symmetric $R$-$R$-bimodule via
left and right multiplication. We also abbreviate and write
$\HH_j(R;M)$ instead of $\HH_j(R/\Z;M)$. If the $S$-algebra $R$ is
commutative, then an $R$-module $N$ determines and is determined
by a unique $R$-symmetric $R$-$R$-bimodule $\tilde{N}$ with
underlying $R$-module $N$. In this situation, we abuse notation
and write $\HH_j(R/S;N)$ instead of $\HH_j(R/S;\tilde{N})$.

We recall from~\cite[Theorem~0.1]{gellerweibel} that, for
all integers $j$, the map
$$\xymatrix{
{ \HH_j(S) \otimes_ST } \ar[r] &
{ \HH_j(T) } \cr
}$$
induced by extension of scalars along $f \colon S \to T$ is an
isomorphism. More generally, for every $S$-module $N$, the map
induced by extension of scalars along $f \colon S \to T$,
$$\xymatrix{
{ \HH_j(S;N) \otimes_ST } \ar[r] &
{ \HH_j(T;N \otimes_ST), } \cr
}$$
is an isomorphism, for all integers $j$. Indeed, by choosing a simplicial
resolution of the $S$-module $N$ by free $S$-modules, we are reduced
to the case $N = S$. We conclude
from~\cite[Theorem~3.1]{lindenstrauss} that, similarly, the $T$-linear
map
$$\xymatrix{
{ \THH_j(S;N) \otimes_ST } \ar[r] &
{ \THH_j(T;N \otimes_ST) } \cr
}$$
induced by extension of scalars along $f \colon S \to T$ is an
isomorphism, for all integers $j$. It also follows from loc.~cit.~that
there is a spectral sequence of $S$-modules
$$\begin{xy}
(0,0)*+{ E_{i,j}^2 = \HH_i(A/S;\THH_j(S;A)) }="1";
(38,0)*+{ \THH_{i+j}(A), }="2";
{ \ar@{=>} "2";"1";};
\end{xy}$$
where the $S$-symmetric $A$-$A$-bimodule structure on $\THH_j(S;A)$ is
induced from left and right multiplication by $A$ on itself. We now extend
scalars along the flat morphism $f \colon S \to T$ to obtain a
spectral sequence of $T$-modules
$$\begin{xy}
(0,0)*+{ E_{s,t}^2 = \HH_i(A/S;\THH_j(S;A)) \otimes_ST
}="1";
(45,0)*+{ \THH_{i+j}(A) \otimes_ST, }="2";
{ \ar@{=>} "2";"1";};
\end{xy}$$
which we compare to the spectral sequence of $T$-modules
$$\begin{xy}
(0,0)*+{ E_{i,j}^2 = \HH_i(A \otimes_ST/T; \THH_j(T;A \otimes_ST))
}="1";
(50,0)*+{ \THH_{i+j}(A \otimes_ST), }="2";
{ \ar@{=>} "2";"1";};
\end{xy}$$
which also is an instance of loc.~cit. The map in the statement
induces a map of spectral sequences which, on $E^2$-terms, is the
composition of the isomorphism
$$\xymatrix{
{ \HH_i(A/S;\THH_j(S;A)) \otimes_ST } \ar[r] &
{ \HH_i(A \otimes_ST/T;\THH_j(S;A) \otimes_ST) } \cr
}$$
obtained by applying the exact functor $-\otimes_ST$ degreewise in the
Hochschild complex and the isomorphism obtained by applying $\HH_i(A
\otimes_ST/T;-)$ to the isomorphism
$$\xymatrix{
{ \THH_j(S;A) \otimes_ST } \ar[r] &
{ \THH_j(T;A \otimes_ST) } \cr
}$$
of $T$-symmetric $A \otimes_ST$-$A \otimes_ST$-bimodules. This
completes the proof.
\end{proof}

\begin{addendum}\label{add:trquasicoherent}Let $S$ be a commutative ring,
let $f \colon S \to T$ be an \'{e}tale morphism of commutative rings,
and let $A$ be any unital associative $S$-algebra, not necessarily
commutative. In this case, the map induced by extension of
scalars along $f \colon S \to T$,
$$\xymatrix{
{ \TR_j^n(A;p) \otimes_{W_n(S)} W_n(T) } \ar[r] &
{ \TR_j^n(A \otimes_ST;p), } \cr
}$$
is an isomorphism, for all prime numbers $p$ and integers
$n \geqslant 1$ and $j$.
\end{addendum}

\begin{proof}The proof is by induction on $n \geqslant 1$ with the case
$n = 1$ being already proved in
Proposition~\ref{prop:quasicoherent}. To prove the induction step, we
use the following diagram
$$\begin{xy}
(-27,28)*{ \vdots };
(-27,24)*{}="11";
(27,28)*{ \vdots };
(27,24)*{}="12";
(-27,14)*+{ \pi_j(\THH(A)_{hC_{\smash{p^{n-1}}}}) \otimes_{W_n(S)} W_n(T) }="21";
(27,14)*+{ \pi_j(\THH(A \otimes_ST)_{h C_{\smash{p^{n-1}}}}) }="22";
(-27,0)*+{ \TR_j^n(A;p) \otimes_{W_n(S)}W_n(T) }="31";
(27,0)*+{ \TR_j^n(A \otimes_ST;p) }="32";
(-27,-14)*+{ \TR_j^{n-1}(A;p) \otimes_{W_n(S)}W_n(T) }="41";
(27,-14)*+{ \TR_j^{n-1}(A \otimes_ST;p) }="42";
(-27,-26.5)*{ \vdots };
(-27,-24)*{}="51";
(27,-26.5)*{ \vdots };
(27,-24)*{}="52";
{ \ar^-{\partial} "21";"11";};
{ \ar^-{\partial} "22";"12";};
{ \ar "22";"21";};
{ \ar^-{N} "31";"21";};
{ \ar^-{N} "32";"22";};
{ \ar "32";"31";};
{ \ar^-{R} "41";"31";};
{ \ar^-{R} "42";"32";};
{ \ar "42";"41";};
{ \ar^-{\partial} "51";"41";};
{ \ar^-{\partial} "52";"42";};
\end{xy}$$
The right-hand column is the long exact sequence of
homotopy groups induced by the ``fundamental'' cofibration sequence 
for the cyclotomic spectrum $\THH(A \otimes_ST)$. It is a sequence
of $W_n(T)$-modules, by the argument in~\cite[pp.~71--72]{hm}. 
The left-hand column is obtained from the corresponding sequence for
$\THH(A)$ by extension of scalars along
$W_n(f) \colon W_n(S) \to W_n(T)$ and it is exact, 
since the ring homomorphism $W_n(f) \colon W_n(S) \to W_n(T)$ again is
\'{e}tale and hence flat by~\cite[Theorem~B]{borger}. Moreover, the
bottom groups $\TR_j^{n-1}(A;p)$ and
$\TR_j^{n-1}(A \otimes_ST;p)$ are considered a $W_n(S)$-module and a
$W_n(T)$-module, respectively, via the respective restriction
maps. Therefore, by~\cite[Corollary~15.4]{borger1}, the bottom
horizontal map agrees with the map
$$\xymatrix{
{ \TR_j^{n-1}(A;p) \otimes_{W_{n-1}(S)} W_{n-1}(T) } \ar[r] &
{ \TR_j^{n-1}(A \otimes_ST;p) } \cr
}$$
that we inductively assume to be an isomorphism. Hence, it will suffice to
show that the top horizontal map is an isomorphism. The argument
in~\cite[pp.~71--72]{hm} also shows that there is a natural spectral
sequence of $W_n(S)$-modules
$$\begin{xy}
(0,0)*+{ E_{i,j}^2 = H_i(C_{\smash{p^{n-1}}}, (F^{n-1})_*(\THH_j(A)))
}="1";
(50,0)*+{ \pi_{i+j}(\THH(A)_{hC_{\smash{p^{n-1}}}} \!) {}^{\phantom{2}} }="2";
{ \ar@{=>} "2";"1";};
\end{xy}$$
from the group homology of $C_{\smash{p^{n-1}}}$ with coefficients in
the $W_n(S)$-module obtained from the $S$-module $\THH_j(A)$
by restriction of scalars along $F^{n-1} \colon W_n(S) \to S$. By
cobase-change along the flat ring homomorphism $W_n(f) \colon W_n(S)
\to W_n(T)$, this gives a spectral sequence of $W_n(T)$-modules
converging to the domain of the top horizontal map in the diagram
above. The target of this map, in turn, is the abutment of the
spectral sequence of $W_n(T)$-modules
$$\begin{xy}
(0,0)*+{ E_{i,j}^2 = H_i(C_{\smash{p^{n-1}}},(F^{n-1})_*(\THH_j(A \otimes_ST))) }="1";
(58,0)*+{ \pi_{i+j}(\THH(A \otimes_ST)_{hC_{\smash{p^{n-1}}}}\!), {}^{\phantom{2}} }="2";
{ \ar@{=>} "2";"1";};
\end{xy}$$
and the map in question induces a map from the
former spectral sequence to the latter which, on $E^2$-terms, is
the $W_n(T)$-linear map
$$\xymatrix{
{ (F^{n-1})_*(\THH_j(A)) \otimes_{W_n(S)}W_n(T) } \ar[r] &
{ (F^{n-1})_*(\THH_j(A \otimes_ST)) } \cr
}$$
induced by extension of scalars along $f \colon S \to T$. 
But~\cite[Corollary~15.4]{borger1} shows that
$$\begin{xy}
(-12,7)*+{ W_n(S) }="11";
(12,7)*+{ S }="12";
(-12,-7)*+{ W_n(T) }="21";
(12,-7)*+{ T }="22";
{ \ar^-{F^{n-1}} "12";"11";};
{ \ar^-{W_n(f)} "21";"11";};
{ \ar^-{f} "22";"12";};
{ \ar^-{F^{n-1}} "22";"21";};
\end{xy}$$
is a cocartesian diagram of commutative rings, so
Proposition~\ref{prop:quasicoherent} implies that the this map is an
isomorphism. This completes the proof.
\end{proof}

A Galois extension of commutative rings is a pair
$(f,\rho)$ of a faithfully \'{e}tale ring homomorphism $f \colon S \to
T$ and a left action $\rho \colon G \to \Aut_S(T)$ by a finite group
$G$ on $T$ through $S$-algebra isomorphisms such that the ring
homomorphism
$$\xymatrix{
{ T \otimes_ST } \ar[r]^-{h} &
{ \prod_{g \in G} T } \cr
}$$
defined by $h(t_1 \otimes t_2) = (g(t_1)t_2)_{g \in G}$ is an
isomorphism. This notion is a special case of the general
notion of a torsor in a topos~\cite[Chapitre~III,
D\'{e}finition~1.4.1]{giraud}.

\begin{corollary}\label{cor:tracyclic}Suppose that
$(f \colon S \to T, \rho \colon G \to \Aut_S(T))$ is a Galois
extension of commutative rings and that $A$ is a unital associative
$S$-algebra, not necessarily commutative. For all prime numbers $p$
and integers $n \geqslant 1$ and $j$, the $W_n(S)$-linear map induced
by extension of scalars along $f \colon S \to T$,
$$\xymatrix{
{ \TR_j^n(A;p) } \ar[r]^-{f^*} &
{ H^0(G,\TR_j^n(A \otimes_ST;p)), } \cr
}$$
is an isomorphism and the $W_n(S)$-modules
$H^i(G,\TR_j^n(A \otimes_ST;p))$ with $i > 0$ vanish.
\end{corollary}

\begin{proof}Let $p$ be a prime number and let $n$ be a positive
integer. We claim that
$$\begin{xy}
(0,0)*+{ ( W_n(S) }="1";
(24,0)*+{ W_n(T), G }="2";
(55,0)*+{ \Aut_{W_n(S)}(W_n(T))) }="3";
{ \ar^-{W_n(f)} "2";"1";};
{ \ar^-{W_n \circ \rho} "3";"2";};
\end{xy}$$
is a Galois extension of commutative rings. The ring homomorphism
$W_n(f)$ is \'{e}tale, by~\cite[Theorem~B]{borger}, and faithful, by
op.~cit.,~Proposition~6.9, and we now consider the commutative diagram
$$\begin{xy}
(-18,7)*+{ W_n(T) \otimes_{W_n(S)}W_n(T) }="11";
(18,7)*+{ \prod_{g \in G} W_n(T) }="12";
(-18,-7)*+{ W_n(T \otimes_S T) }="21";
(18,-7)*+{ W_n(\prod_{g \in G} T) }="22";
{ \ar^-{h} "12";"11";};
{ \ar "21";"11";};
{ \ar "12";"22";};
{ \ar^-{W_n(h)} "22";"21";};
\end{xy}$$
in which the vertical maps are the canonical maps. It is proved in
op.~cit., Corollary~9.4, that the left-hand vertical map is an
isomorphism, and it follows immediately from the definition of Witt
vectors that the right-hand vertical map is an isomorphism. Finally,
by assumption, the lower horizontal map is an isomorphism, and hence,
the top horizontal map is an isomorphism, as desired. 

Finally, we abbreviate $M = \TR_j^n(A;p)$, $k = W_n(S)$,
and $R = W_n(T)$, and consider the augmented cosimplicial $k$-module
$$\xymatrix{
{ M } \ar[r]^-{\eta} &
{ M \otimes_kR^{\,\otimes_k[-]} }. \cr
}$$
Since $k \to R$ is faithfully flat, this map is a weak equivalence, by
faithfully flat descent for modules. We also consider the augmented
simplicial $k[G]$-module
$$\xymatrix{
{ k[G]^{\,\otimes_k[-]} \otimes_k N }
\ar[r]^-{\mu} &
{ N, }
}$$
which is a weak equivalence for every $k[G]$-module $N$. Now, since
$k \to R$ is Galois with group $G$, we have an isomorphism of
cosimplicial $k$-modules
$$\xymatrix{
{ M \otimes_kR^{\,\otimes_k[-]} } \ar[r] &
{ \Hom_{k[G]}(k[G]^{\otimes_k[-]}, M \otimes_kR). } \cr
}$$
Hence, we conclude that there is a canonical isomorphism
$$\xymatrix{
{ M } \ar[r] &
{ R\Hom_{k[G]}(k,M \otimes_kR), }
}$$
in the derived category of $k$-modules. This proves the corollary.
\end{proof}

\section{Reduced trace isomorphisms}\label{sec:Trd}

We now prove the first part of Theorem~\ref{thm:theoremB}, which is
equivalent to the statement that there exists an element
$\tilde{y} \in \TC_0(A\,|\,D,\Zp)$ whose image
$y_1 \in \THH_0(A\,|\,D,\Zp)$ freely generates
$\THH_*(A\,|\,D,\Zp)$ as a graded $\THH_*(S\,|\,K,\Zp)$-module.

\begin{lemma}\label{lem:conjugation}Let $\pi \in A$ be a
non-zero element. Conjugation by $\pi$ in $D$ defines an automorphism
$\alpha \colon A \to A$ and the endomorphism of the cyclotomic spectrum
$$\THH(A\,|\,D) =
\THH(\Ch^b(\mathcal{P}_A, \mathcal{M}_A,\mathcal{T}_A))$$
induced by extension of scalars along $\alpha$ is homotopic to the identity.
\end{lemma}

\begin{proof}Since $v_D(\pi x \pi^{-1}) = v_D(x)$, conjugation
by $\pi$ in $D$ restricts to an automorphism $\alpha$ of $A$, as
stated. Extension of scalars along $\alpha$ defines an exact functor
$$\begin{xy}
(0,0)*+{ \Ch^b(\mathcal{P}_A, \mathcal{M}_A, \mathcal{T}_A) }="1";
(38,0)*+{ \Ch^b(\mathcal{P}_A, \mathcal{M}_A, \mathcal{T}_A) }="2";
{ \ar^-{\alpha^*} "2";"1";};
\end{xy}$$
and left multiplication by $\pi$ defines a natural transformation
$h \colon \alpha^* \Rightarrow \id$. Indeed, if $P$ is in
$\mathcal{P}_A$, then $\alpha^*(P) = A \otimes_AP$, where $A$ is
considered as a right $A$-module via $\alpha$, and the map
$h_P \colon A \otimes_A P \to P$, which maps
$1 \otimes x$ to $\pi x$ is $A$-linear, since
$$h_P(a \otimes x)
= h_P(1 \otimes \pi^{-1}a\pi x) 
= \pi \pi^{-1}a\pi x 
= a\pi x
= a h_P(x).$$
Since $\pi \in A$ is a non-zero-divisor, the natural transformation
$h \colon \alpha^* \Rightarrow \id$ is exact, and therefore, induces a
homotopy through maps of cyclotomic spectra from the map induced by
$\alpha^*$ to the identity map.
\end{proof}

\begin{addendum}\label{add:conjugation}The endomorphism of
$\THH(A \otimes_ST\,|\,D\otimes_KL)$ induced by extension of scalars
along $\alpha \otimes \id$ is homotopic to the identity, as a map of
cyclotomic spectra.
\end{addendum}

\begin{proof}The proof is entirely analogous to the proof of
Lemma~\ref{lem:conjugation}. That the natural transformation $h \colon
(\alpha \otimes \id)^* \Rightarrow \id$ is exact again uses that $\pi
\otimes 1$ is a non-zero-divisor in the regular ring $A \otimes_ST$.
\end{proof}

In particular, conjugation in $D$ by our chosen generator
$\pi_D \in \mathfrak{m}_D \subset A$ defines an automorphism
$\sigma \colon A \to A$. Moreover, since $\pi_D^{\,d} \in K$ is central,
we see that the action of $G = \Gal(L/K)$ on $T/S$ extends to an
action on $A/S$. This, in turn, induces an action of $G$ on
the cyclotomic spectrum $\THH(A\,|\,D)$, and by
Lemma~\ref{lem:conjugation}, the action by the
generator $\sigma \in G$ is homotopic to the identity map. It follows
that $G$ acts trivially on the homotopy groups 
$\THH_*(A\,|\,D)$, $\TR_*^n(A\,|\,D;p)$, $\TC_*(A\,|\,D)$,
etc.\footnote{\,By contrast, we show in Corollary~\ref{cor:nontrivialaction} below
  that if $p$ divides $d$, then the $G$-action on the
  cyclotomic spectrum $\THH(A\,|\,D)$ is non-trivial.}

\begin{theorem}\label{thm:trreducedtrace}For all $n \geqslant 1$, the
graded $\TR_*^n(S\,|\,K;p)$-module $\TR_*^n(A\,|\,D;p)$ is free of
rank one on a canonical generator $y_n$ of degree zero.
\end{theorem}

\begin{proof}We consider the diagram of exact functors
$$\begin{xy}
(0,7)*+{ \Ch^b(\mathcal{P}_A, \mathcal{M}_A, \mathcal{T}_A) }="11";
(45,7)*+{ \Ch^b(\mathcal{P}_{A \otimes_ST}, \mathcal{M}_{A
    \otimes_ST}, \mathcal{T}_{A \otimes_ST}) }="12";
(0,-7)*+{ \Ch^b(\mathcal{P}_S, \mathcal{M}_S, \mathcal{T}_S)  }="21";
(45,-7)*+{ \Ch^b(\mathcal{P}_T, \mathcal{M}_T, \mathcal{T}_T), }="22";
{ \ar^-{f^*} "12";"11";};
{ \ar^-{f^*} "22";"21";};
{ \ar@<-.7ex>_-{\Trd_{A\otimes_ST/T}} "22";"12";};
{ \ar@<-.7ex>_-{\Ird_{A\otimes_ST/T}} "12";"22";};
\end{xy}$$
that we defined in the introduction. The group $G$ acts on $A$ and
$T$, and we let it act diagonally on $A \otimes_ST$ and trivially on $S$.
With respect to these actions, all functors in the diagram are
$G$-equivariant, and hence, we obtain the following induced diagram of
$\THH(S\,|\,K)$-modules in cyclotomic spectra with $G$-action,
$$\begin{xy}
(0,7)*+{ \THH(A\,|\,D) }="11";
(35,7)*+{ \THH(A \otimes_ST \,|\, D \otimes_KL) }="12";
(0,-7)*+{ \THH(S\,|\,K) }="21";
(35,-7)*+{ \THH(T\,|\,L). }="22";
{ \ar^-{f^*} "12";"11";};
{ \ar@<-.7ex>_-{\Trd_{A \otimes_ST/T}} "22";"12";};
{ \ar@<-.7ex>_-{\Ird_{A \otimes_ST/T}} "12";"22";};
{ \ar^-{f^*} "22";"21";};
\end{xy}$$
In this diagram, the right-hand vertical maps both are equivalences, since
the counit and the unit of the adjunction
$(\Trd_{A \otimes_ST/T}, \Ird_{A   \otimes_ST/T}, \epsilon, \eta)$
both are exact. While the $G$-action on the botton left-hand term is
trivial, this is generally not true for the upper left-hand term. However, by
Lemma~\ref{lem:conjugation}, the action is trivial on homotopy groups,
so we obtain the following induced diagram of graded
$\TR_*^n(S\,|\,K;p)$-modules, 
$$\begin{xy}
(0,7)*+{ \TR_*^n(A\,|\,D;p) }="11";
(43,7)*+{ H^0(G,\TR_*^n(A \otimes_ST \,|\, D \otimes_KL;p)) }="12";
(0,-7)*+{ \TR_*^n(S\,|\,K;p) }="21";
(43,-7)*+{ H^0(G,\TR_*^n(T\,|\,L;p)). }="22";
{ \ar^-{f^*} "12";"11";};
{ \ar@<-.7ex>_-{\Trd_{A \otimes_ST/T}} "22";"12";};
{ \ar@<-.7ex>_-{\Ird_{A \otimes_ST/T}} "12";"22";};
{ \ar^-{f^*} "22";"21";};
\end{xy}$$
In this diagram, the horizontal maps are isomorphisms by
Corollary~\ref{cor:tracyclic}, and the vertical maps are isomorphisms by
what was just said. Moreover, by Skolem-Noether, all $K$-algebra
homomorphisms $L \to D$ are conjugate in $D$, and therefore, it
follows from Addendum~\ref{add:conjugation} that the vertical maps
are independent of the choice of
maximal unramified subfield $L \subset D$. Hence, we obtain canonical
inverse isomorphisms
$$\begin{xy}
(0,0)*+{ \TR_*^n(A\,|\,D;p) }="1";
(33,0)*+{ \TR_*^n(S\,|\,K;p) }="2";
{ \ar@<.7ex>^-{\Trd_{A/S}} "2";"1";};
{ \ar@<.7ex>^-{\Ird_{A/S}} "1";"2";};
\end{xy}$$
given by the maps making the respective square diagrams
commute. Equivalently, the graded $\TR_*^n(S\,|\,K;p)$-module
$\TR_*^n(A\,|\,D;p)$ is free on the single canonical generator
$y_n = \Ird_{A/S}(1) \in \TR_0^n(A\,|\,D;p)$, as stated.
\end{proof}

Next, we wish to prove the analogous result for the $p$-adic homotopy
groups. In general, the $p$-adic homotopy groups of a spectrum $X$ are
defined to the homotopy groups of its $p$-completion, and
by~\cite[Proposition~2.5]{bousfield} there is an exact sequence
$$0 \to
\Ext_{\Z}^1(\Q_p/\Zp, \pi_j(X)) \to
\pi_j(X,\Zp) \to
\Hom_{\Z}(\Q_p/\Zp, \pi_{j-1}(X)) \to
0$$
that relates the homotopy groups of $X$ and the $p$-adic homotopy
groups of $X$.

\begin{definition}\label{def:pcontrolled}An abelian group $M$ is
\emph{$p$-controlled} if it is a direct sum of a uniquely divisible
$M_{\dv}$ and a group $M_{\tor}$ annihilated by some power of $p$. 
\end{definition}

\begin{lemma}\label{lem:cont}If $M$ and $N$ are $p$-controlled, then
every extension of $N$ by $M$ and the kernel and cokernel of every
homomorphism $M\to N$ are also $p$-controlled. If $M$ is
$p$-controlled, then $\Hom_{\Z}(M,\Z_p)$ and $\Hom_{\Z}(\Q_p/\Z_p,M)$
both vanish.
\end{lemma}

\begin{proof}The only statement that needs proof is that the
full subcategory of the abelian category of abelian groups is closed
under extensions. So we let
$$\xymatrix{
{ 0 } \ar[r] &
{ M } \ar[r] &
{ P } \ar[r] &
{ N } \ar[r] &
{ 0 } \cr
}$$
be an exact sequence of abelian groups, where $M$ and $N$ are
$p$-controlled, and wish to show that $P$ is $p$-controlled. As
uniquely divisible groups are $\Z$-injective, $P$ is isomorphic to
$M_{\dv} \oplus P/M_{\dv}$, so replacing $M$ and $P$ by $M_{\tor}$ and
$P/M_{\dv}$, respectively, we assume that $M = M_{\tor}$. Let
$P^{\tor}$ and $P^{\dv}$ be the inverse images of $N_{\tor}$ and
$N_{\dv}$ in $P$, respectively. It suffices to prove that $P^{\tor}$ and
$P^{\dv}$ are $p$-controlled, so we may reduce to the two cases
$N=N_{\tor}$ and $N = N_{\dv}$. The first is trivial, and for the
second, it suffices to prove that $\Ext_{\Z}^1(N,M) = 0$. But this is
clear, since this group is both $p$-divisible and killed by a power of $p$.
\end{proof}

\begin{proposition}\label{prop:trpcontrolled}The group
$\TR^n_j(S\,|\,K;p)$ is $p$-controlled, for all $n\geqslant 1$ and
$j \geqslant 1$. 
\end{proposition}

\begin{proof}The groups $\TR_j^n(k_S;p)$ are finite $p$-groups
for all integers $n \geqslant 1$ and $j$. Hence, by
Proposition~\ref{prop:thhlocalizationsequence} and
Lemma~\ref{lem:cont}, it suffices to show that $\TR_j^n(S;p)$ is
$p$-controlled, for all $n \geqslant 1$ and $j \geqslant 1$. If $S$ is
of equal characteristic, then the groups $\TR_j^n(S;p)$ are
annihilated by $p^n$. So we assume that $S$ is of mixed characteristic
and proceed by induction on $n \geqslant 1$. In the case $n=1$, we use
the spectral sequence
$$E_{i,j}^2 = \HH_i(S/\Zp, \THH_j(\Z_p) \otimes_{\Z_p} S) \Rightarrow
\THH_{i+j} (S)$$ 
from~\cite[Corollary 3.3]{lindenstrauss}. If $j \geqslant 1$, then
$\THH_j(\Zp)$ is $p$-controlled, by Theorem~2.2 and
Example~3.4 of \cite{larsenlindenstrauss}, and hence, so are the
groups $E_{i,j}^2$, since $\HH_i(S/\Zp,M)$ is $M$, for $i = 0$, and is
annihilated by a fixed power of $p$, for $i > 0$. This also shows that
$E_{i,0}^2$ is $p$-controlled, for $i > 0$. Therefore,
Lemma~\ref{lem:cont} and the spectral sequence shows that $\THH_j(S)$
is $p$-controlled.

To prove the induction step, we use the ``fundamental'' cofibration
$$\xymatrix{
{ \THH(S)_{hC_{\smash{p^{n-1}}}} } \ar[r]^-{N} &
{ \TR^n(S;p) } \ar[r]^-{R} &
{ \TR^{n-1}(S;p) } \cr
}$$
and the spectral sequence
$$E_{i,j}^2 = H_i(C_{\smash{p^{n-1}}},\THH_j(S)) \Rightarrow
\pi_{i+j}(\THH(S)_{hC_{\smash{p^{n-1}}}}).$$
By the case $n = 1$, we conclude that $E_{i,j}^2$ is $p$-controlled,
if $i > 0$ or $j > 0$, and hence, Lemma~\ref{lem:cont} shows that
$\pi_j(\THH(S)_{hC_{\smash{p^{n-1}}}})$ is $p$-controlled, for
$j > 0$. The induction step now follows from Lemma~\ref{lem:cont},
since the boundary map
$$\xymatrix{
{ \TR_1^{n-1}(S;p) } \ar[r]^-{\partial} \ar[r]^-{\partial} &
{ \pi_0(\THH(S)_{hC_{\smash{p^{n-1}}}}) } \cr
}$$
is zero for every commutative ring.
\end{proof}

\begin{addendum}\label{add:trreducedtrace}The graded
$\TR_*^n(S\,|\,K;p,\Zp)$-module $\TR_*^n(A\,|\,D;p,\Zp)$ is free of
rank one with canonical generator $y_n \in \TR_0^n(A\,|\,D;p,\Zp)$, for
all $n \geqslant 1$.
\end{addendum}

\begin{proof}If $S$ is of equal characteristic, then 
$\TR^n(S\,|\,K;p)$ and $\TR^n(A\,|\,D;p)$ are already $p$-complete,
so there is nothing further to prove. If $S$ is of mixed
characteristic, then by Proposition~\ref{prop:trpcontrolled} and
Lemma~\ref{lem:cont}, the horizontal maps in the diagram
$$\begin{xy}
(0,7)*+{ \Ext_{\Z}^1(\Q_p/\Zp,\TR_*^n(A\,|\,D;p)) }="11";
(43,7)*+{ \TR_*^n(A\,|\,D;p,\Zp)) }="12";
(0,-7)*+{ \Ext_{\Z}^1(\Q_p/\Zp,\TR_*^n(S\,|\,K;p)) }="21";
(43,-7)*+{ \TR_*^n(S\,|\,K;p,\Zp)) }="22";
{ \ar "12";"11";};
{ \ar "22";"21";};
{ \ar@<-.7ex>_-{\Trd_{A/S}} "21";"11";};
{ \ar@<-.7ex>_-{\Trd_{A/S}} "22";"12";};
{ \ar@<-.7ex>_-{\Ird_{A/S}} "11";"21";};
{ \ar@<-.7ex>_-{\Ird_{A/S}} "12";"22";};
\end{xy}$$
of graded $\TR_*^n(S\,|\,K;p,\Zp)$-modules are isomorphisms, and we
define the right-hand vertical maps to be the unique maps that make
the respective square diagrams commute. Finally, by
Theorem~\ref{thm:trreducedtrace}, the left-hand vertical maps are
mutally inverse isomorphisms, and hence, so are the right-hand vertical
maps.
\end{proof}

\begin{lemma}\label{lem:mittagleffler}For all $j \geqslant 1$, the
limit systems
$$\begin{aligned}
{} & \xymatrix{
{ \cdots } \ar[r]^-{R} &
{ \TR_j^n(S;p,\Zp) } \ar[r]^-{R} &
{ \;\cdots\; } \ar[r]^-{R} &
{ \TR_j^2(S;p,\Zp) } \ar[r]^{R} &
{ \TR_j^1(S;p,\Zp) } \cr
} \cr
{} & \xymatrix{
{ \cdots } \ar[r]^-{F} &
{ \TR_j^n(S;p,\Zp) } \ar[r]^-{F} &
{ \;\cdots\; } \ar[r]^-{F} &
{ \TR_j^2(S;p,\Zp) } \ar[r]^{F} &
{ \TR_j^1(S;p,\Zp) } \cr
} \cr
\end{aligned}$$
both satisfy the Mittag-Leffler condition.
\end{lemma}

\begin{proof}If $S$ is of mixed characteristic, then the groups
$\TR_j^n(S;p,\Zp)$ with $j \geqslant 1$ are all finite
$p$-groups. Indeed, the case $n = 1$ is proved
in~\cite{lindenstraussmadsen}, and the general case follows by an
inductive argument similar to the proof of
Proposition~\ref{prop:trpcontrolled} above. In particular, the limit
systems satisfy the Mittag-Leffler condition. If $S$ is of equal
characteristic, then it follows from~\cite[Theorem~B]{h} that the
canonical map
$$\xymatrix{
{ W_n\Omega_S^* \otimes \TR_*^n(\Fp;p) } \ar[r] &
{ \TR_*^n(S;p) } \cr
}$$
is an isomorphism. Indeed, the ring $S = k_S[\![t]\!]$ is a regular
$k_S$-algebra, and hence, is a filtered colimit of smooth
$k_S$-algebras by~\cite{popescu}. Moreover, by B\"{o}kstedt
periodicity, we have $\TR_*^n(\Fp;p) = \Z/\!p^n\,[x_n]$, where $x_n$ has
degree $2$ and may be chosen such that $R(x_n) = px_{n-1}$ and
$F(x_n) = x_{n-1}$. The structure of the de~Rham-Witt groups was
determined in~\cite[Theorem~B]{gh1}. It shows in particular that
$$\begin{aligned}
{} \im( \! \xymatrix{
{ \lim_{m,R} \TR_j^m(S;p) } \ar[r]^-{\pr} &
{ \TR_j^n(S;p) } \cr
} \! ) 
{} & \,=\, \im( \! \xymatrix{
{ \TR_j^{2n}(S;p) } \ar[r]^-{R^n} &
{ \TR_j^n(S;p) } \cr
} \! ), \cr
{} \im( \! \xymatrix{
{ \lim_{m,F} \TR_j^m(S;p) } \ar[r]^-{\pr} &
{ \TR_j^n(S;p) } \cr
} \! ) 
{} & \,=\, \im( \! \xymatrix{
{ \TR_j^{2n}(S;p) } \ar[r]^-{F^n} &
{ \TR_j^n(S;p) } \cr
} \! ), \cr
\end{aligned}$$
whence the lemma.
\end{proof}

\begin{theorem}\label{thm:trtf}The graded
  $\TR_*(S\,|\,K;p,\Zp)$-module $\TR_*(A\,|\,D;p,\Zp)$ is free 
 on a canonical generator $y$. The graded
  $\TF_*(S\,|\,K;p,\Zp)$-module $\TF_*(A\,|\,D;p,\Zp)$ is free
  on a canonical generator $y$. 
\end{theorem}

\begin{proof}We prove the second statement; the first statement is
proved analogously. By Theorem~\ref{thm:trreducedtrace}, the graded
$\TR_*^n(S\,|\,K;p,\Zp)$-module $\TR_*^n(A\,|\,D;p,\Zp)$ is free on
the canonical generator $y_n = \Ird_{A/S}(1) \in
\TR_0^n(A\,|\,D;p,\Zp)$. Moreover, we have
$$F(y_n) = F(\Ird_{A/S}(1)) = \Ird_{A/S}(F(1)) = \Ird_{A/S}(1) =
y_{n-1}.$$
Indeed, the second identity holds, since the maps
\vspace{-1mm}
$$\begin{xy}
(0,7)*+{ \THH(A\,|\,D) }="11";
(35,7)*+{ \THH(A \otimes_ST \,|\, D \otimes_KL) }="12";
(0,-7)*+{ \THH(S\,|\,K) }="21";
(35,-7)*+{ \THH(T\,|\,L). }="22";
{ \ar^-{f^*} "12";"11";};
{ \ar@<-.7ex>_-{\Trd_{A \otimes_ST/T}} "22";"12";};
{ \ar@<-.7ex>_-{\Ird_{A \otimes_ST/T}} "12";"22";};
{ \ar^-{f^*} "22";"21";};
\end{xy}$$
that we used to define the isomorphism $\Ird_{A/S}$ are maps of
cyclotomic spectra, and the third identity holds, since $F$ is a ring
homomorphism. Finally, we claim that the canonical maps
$$\begin{aligned}
{} & \xymatrix{
{ \TF_j(A\,|\,D;p,\Zp) } \ar[r] &
{ \lim_{n,F} \TR_j^n(A\,|\,D;p,\Zp), } \cr
} \cr
{} & \xymatrix{
{ \TF_j(S\,|\,K;p,\Zp) } \ar[r] &
{ \lim_{n,F} \TR_j^n(S\,|\,K;p,\Zp)\phantom{,} } \cr
} \cr
\end{aligned}$$
are isomorphisms, or equivalently, that the corresponding derived
limits vanish. In the case $j \geqslant 1$, this follows from
Lemma~\ref{lem:mittagleffler}, and in the case $j = 0$, it follows
from the fact that $\TR_0^n(S\,|\,K;p,\Zp) = W_n(S)$ admits a compact
topology with respect to which the map $F$ is continuous. So the
theorem follows with $y \in \TF_0(A\,|\,D;p,\Zp)$ the unique class
with image $y_n \in \TR_0^n(A\,|\,D;p,\Zp)$ for all $n \geqslant 1$.
\end{proof}

\begin{proof}[Proof of Theorem~\ref{thm:theoremB}~{\rm (1)}]We recall
that for a cyclotomic spectrum $X$, its topological cyclic homology is
equivalently given as the homotopy equalizer 
$$\xymatrix{
{ \TC(X;p) } \ar[r]^-{i} &
{ \TF(X;p) } \ar@<.7ex>[r]^-{\varphi^{-1}} \ar@<-.7ex>[r]_-{\id} &
{ \TF(X;p), } \cr
}$$
where the inverse Frobenius $\varphi^{-1}$ is defined to be the composition
$$\xymatrix{
{ \holim_{n,F} \TR^n(X;p) } \ar[r]^-{\operatorname{res}_\sigma} &
{ \holim_{n,F} \TR^{n+1}(X;p) } \ar[r]^-{R} &
{ \holim_{n,F} \TR^n(X;p) } \cr
}$$
of the restriction along the successor functor followed by the map of
limits induced by the restriction maps. Hence, in the induced diagram
$$\xymatrix{
{ \TC_0(X;p) } \ar[r]^-{i} &
{ \TF_0(X;p) } \ar@<.7ex>[r]^-{\varphi^{-1}} \ar@<-.7ex>[r]_-{\id} &
{ \TF_0(X;p), } \cr
}$$
the map $i$ surjects onto the equalizer of maps $\varphi^{-1}$ and
$\id$. Now, if $X$ is the cyclotomic spectrum $\THH(A\,|\,D,\Zp)$,
then the element $y \in \TF_0(A\,|\,D;p,\Zp)$ satisfies
$\varphi^{-1}(y) = y$, because $R(y_n) = y_{n-1}$. Hence, there exists
$\tilde{y} \in \TC_0(A\,|\,D\;,\Zp)$ such that $i(\tilde{y}) = y$. The
class $\tilde{y}$ defines a component in the mapping space
$$\begin{aligned}
{} & \Map_{\operatorname{Mod}_{\THH(S\,|\,K,\Zp)}
(\mathsf{CycSp}_p)}(\THH(S\,|\,K,\Zp),\THH(A\,|\,D,\Zp)) \cr
{} & \simeq \Map_{\mathsf{CycSp}_p}
(\mathbb{S}_p,\THH(A\,|\,D,\Zp)) \simeq
\Omega^{\infty}\TC(A\,|\,D,\Zp), \cr
\end{aligned}$$
and every point in this component induces the isomorphism
\vspace{-1mm}
$$\begin{xy}
(0,0)*+{ \THH_*(S\,|\,K,\Zp) }="1";
(36,0)*+{ \THH_*(A\,|\,D,\Zp) }="2";
{ \ar^-{\Ird_{A/S}} "2";"1";};
\end{xy}$$
on the level of homotopy groups. Indeed, the class
$\tilde{y} \in \TC_0(A\,|\,D,\Zp)$ maps to the class
$y_1 \in \THH_0(A\,|\,D,\Zp)$ and $\Ird_{A/S}$ is given by
multiplication by $y_1$. This shows that every point in the component
$\tilde{y}$ of the mapping space defines an equivalence of
$\THH(S\,|\,K,\Zp)$-modules in cyclotomic spectra,
\vspace{-1mm}
$$\begin{xy}
(0,0)*+{ \THH(S\,|\,K,\Zp) }="1";
(33,0)*+{ \THH(A\,|\,D,\Zp), }="2";
{ \ar^-{\Ird_{A/S}} "2";"1";};
\end{xy}$$
whose inverse $\Trd_{A/S}$ therefore also is an equivalence of
$\THH(S\,|\,K,\Zp)$-modules in cyclotomic spectra. This proves
part~(1) of the Theorem~\ref{thm:theoremB}.
\end{proof}

\section{Comparison with the trace map}\label{sec:comparison}

The equivalence of cyclotomic spectra in
Theorem~\ref{thm:theoremB}~(1) shows, in particular, that the graded
$\TC_*^{-}(S\,|\,K,\Zp)$-module $\TC_*^{-}(A\,|\,D,\Zp)$ and the
graded $\TP_*(S\,|\,K,\Zp)$-module $\TP_*(A\,|\,D,\Zp)$ both are free
of rank one generated by the image under
\vspace{-1mm}
$$\begin{xy}
(0,0)*+{ \TC_*(A\,|\,D,\Zp) }="1";
(31,0)*+{ \TC_*^{-}(A\,|\,D,\Zp) }="2";
(63,0)*+{ \TP_*(A\,|\,D,\Zp) }="3";
{ \ar^-{i} "2";"1";};
{ \ar@<.8ex>^-{\varphi} "3";"2";};
{ \ar@<-.8ex>_-{\can} "3";"2";};
\end{xy}$$
of the generator $\tilde{y} \in \TC_0(A\,|\,D,\Zp)$. We first use results
from~\cite{gh1} and~\cite{hm4} to prove structural results about
these groups. 

\begin{lemma}\label{lem:frobeniusequivalence}The Frobenius map
$$\xymatrix{
{ \THH(S\,|\,K,\Zp) } \ar[r]^-{\varphi} &
{ \THH(S\,|\,K,\Zp)^{tC_p} } \cr
}$$
induces an equivalence of connective covers.
\end{lemma}

\begin{proof}If $S$ is of mixed characteristic, then this is proved
in~\cite[Theorem~5.4.3]{hm4}. We therefore assume that
$S = k_S[\![\pi]\!]$ is of equal characteristic. Since the
Frobenius $\varphi \colon \THH(k_S) \to \THH(k_S)^{tC_p}$ induces
an equivalence of $(-1)$-connective covers, it suffices to show that
the Frobenius $\varphi \colon \THH(S) \to \THH(S)^{tC_p}$ induces an
equivalence of connective covers. Now, by~\cite[Theorem~B]{h}
and~\cite{popescu}, the canonical map 
$$\xymatrix{
{ \Omega_S^* \otimes \THH_*(\Fp) } \ar[r] &
{ \THH_*(S) } \cr
}$$
is an isomorphism. Since $k_S$ is perfect, $\Omega_S^*$ is an exterior
algebra over $S$ on a generator $d\pi$ of degree $1$, and
$\THH_*(\Fp)$ is a polynomial algebra over $\Fp$ on a generator $x$ of
degree $2$ by B\"{o}kstedt periodicity. Hence, in the Tate spectral sequence
$$E_{i,j}^2 = \hat{H}^{-i}(C_p,\THH_j(S)) \Rightarrow
\pi_{i+j}(\THH(S)^{tC_p}),$$
we have $E^2 = S \otimes \Lambda\{u,d\pi\} \otimes \Fp[t^{\pm 1},x]$
with $\deg(u) = (-1,0)$, $\deg(d\pi) = (0,1)$, $\deg(t) = (-2,0)$, and
$\deg(x) = (0,2)$. The $d^2$-differential is a derivation, which is
automatically $S^p$-linear, and is given by $d^2(\pi) = td\pi$, so
the $E^3$-term takes the form 
$$E^3 = S^p \otimes \Lambda\{u,\pi^{p-1}d\pi\} \otimes
\Fp[t^{\pm1},x].$$
The Frobenius $\varphi$ maps the class $a \in S$ (resp.~$d\pi$,
resp. $x$) to the class represented in the spectral sequence by
$a^p \in S^p$ (resp.~by $\pi^{p-1}d\pi$, resp.~by $t^{-1}$), which
therefore is an infinite cycle. 
%
%
%%%%% This should be justified.
%
%
Moreover, comparing with the spectral
sequence for $\THH(\Fp)$, we see that also $x$ is an infinite cycle
and that, up to a unit in $\Fp$,
$$d^3(u) = t^2x.$$
Hence, all further differentials in the spectral sequence are zero, and
$$E^{\infty} = S^p \otimes \Lambda\{\pi^{p-1}d\pi\} \otimes
\Fp[t^{\pm1}].$$
In particular, the Frobenius $\varphi \colon \THH(S) \to \THH(S)^{tC_p}$
induces an isomorphism of homotopy groups in degrees $j \geqslant 0$,
as desired.
\end{proof}

\begin{corollary}\label{cor:frobeniusequivalence}The Frobenius map
induces an equivalence of connective covers
$$\xymatrix{
{ \TF(S\,|\,K;p,\Zp) } \ar[r] &
{ \TC^{-}(S\,|\,K,\Zp). } \cr
}$$
\end{corollary}

\begin{proof}For every cyclotomic spectrum $X$, we have the canonical
projection
$$\xymatrix{
{ \TR^n(X;p) = X \times_{\textstyle{X^{tC_p}}} X^{hC_p} 
\times_{\textstyle{X^{tC_{\smash{p^2}}}}} \cdots
\times_{\textstyle{X^{tC_{\smash{p^{n-1}}}}}} X^{hC_{\smash{p^{n-1}}}}
} \ar[r]^-{\pr_n} & 
{ X^{hC_{\smash{p^{n-1}}}} } \cr
}$$
and there are commutative diagrams
$$\begin{xy}
(0,7)*+{ \TR^n(X;p) }="11";
(30,7)*+{ X^{hC_{\smash{p^{n-1}}}} }="12";
(0,-7)*+{ \TR^{n-1}(X;p) }="21";
(30,-7)*+{ X^{hC_{\smash{p^{n-2}}}} }="22";
{ \ar^-{\pr_n} "12";"11";};
{ \ar^-{F} "21";"11";};
{ \ar "22";"12";};
{ \ar^-{\pr_{n-1}} "22";"21";};
\end{xy}$$
in which the horizontal maps are the canonical projections and the
right-hand vertical maps are the map induced by the subgroup
inclusions $C_{\smash{p^{n-2}}} \subset C_{\smash{p^{n-1}}}$. Hence,
the canonical projections induce a map of homotopy limits
$$\xymatrix{
{ \TF(X;p) = \holim_{n,F} \TR^n(X;p) } \ar[r] &
{ \holim_n X^{hC_{\smash{p^{n-1}}}}. } \cr
}$$
Moreover, if $X$ is $p$-complete, then the map
$$\xymatrix{
{ \holim_n X^{hC_{\smash{p^{n-1}}}} } &
{ X^{h\mathbb{T}} = \TC^{-}(X), } \ar[l] \cr
}$$
induced by the subgroup inclusions
$C_{\smash{p^{n-1}}} \subset \mathbb{T}$ is an equivalence, so we get
a map
$$\xymatrix{
{ \TF(X;p) } \ar[r] &
{ \TC^{-}(X). } \cr
}$$
As pointed out in the proof of~\cite[Theorem~3]{krause}, it is clear from this
construction that if the Frobenius $\varphi \colon X \to X^{tC_p}$
induces an isomorphism of homotopy groups in degrees $j \geqslant d$,
then so do the maps $\pr_n$ and map above. So by taking $X$ to be the
$p$-complete cyclotomic spectrum $\THH(S\,|\,K,\Zp)$, the corollary
follows from Lemma~\ref{lem:frobeniusequivalence}.
\end{proof}

\begin{lemma}\label{lem:wittvectors}The canonical projection
$i \colon S \to k_S$ induces an isomorphism
$$\begin{xy}
(0,0)*+{ \lim_{n,F} W_n(S) }="1";
(30,0)*+{ \lim_{n,F} W_n(k_S). }="2";
{ \ar "2";"1";};
\end{xy}$$
\end{lemma}

\begin{proof}Let $\mathfrak{m}_S \subset S$ be the maximal ideal. The
injectivity and the surjectivity of the map in the statement are
equivalent to the vanishing of the limit
$\lim_{n,F}W_n(\mathfrak{m}_S)$ and the derived limit
$R^1\lim_{n,F}W_n(\mathfrak{m}_S) = 0$, respectively. We recall that
the Witt vector Frobenius $F \colon W_n(S) \to W_{n-1}(S)$ satisfies
$F(a) = R(a)^p$ modulo $W_{n-1}(pS)$; see for
example~\cite[Lemma~1.8]{h9}. It follows that
$F(W_n(\mathfrak{m}_S^m)) \subset W_{n-1}(\mathfrak{m}_S^{m+1})$, for
all $m > 0$, which shows that the limit vanishes. The derived limit
vanishes, since $W_n(\mathfrak{m}_S)$ has a compact
topology for which $F \colon W_n(\mathfrak{m}_S) \to
W_{n-1}(\mathfrak{m}_S)$ is continuous.
\end{proof}

Since $k_S$ is perfect, we further identify the common ring in
Lemma~\ref{lem:wittvectors} with the ring of Witt vectors $W(k_s)$ via
the isomorphism 
$$\xymatrix{
{ \lim_{n,F} W_n(k_S) } \ar[r] &
{ \lim_{n,R} W_n(k_S) = W(k_S) } \cr
}$$
that at level $n$ is given by the map $W_n(k_S) \to W_n(k_S)$ induced
by $\varphi^n \colon k_S \to k_S$. It is an isomorphism, since $k_S$
is perfect. In particular, the ring
$$\textstyle{ \TF_0(S\,|\,K;p,\Zp) = \lim_{n,F}W_n(S) }$$
is an integral domain.

\begin{theorem}\label{thm:evenisomorphism}For every even integer
$j = 2k > 0$, the map 
\vspace{-1mm}
$$\begin{xy}
(0,0)*+{ \TC_{\smash{j}}^{-}(S\,|\,K,\Zp) }="1";
(33,0)*+{ \TP_j(S\,|\,K,\Zp) }="2";
{ \ar^-{\varphi - \can} "2";"1";};
\end{xy}$$
is an isomorphism.
\end{theorem}

\begin{proof}It follows from
  Lemma~\ref{lem:frobeniusequivalence} and 
the Tate orbit lemma~\cite[Lemma~I.2.1]{nikolausscholze} that
\vspace{-1mm}
$$\begin{xy}
(0,0)*+{ \TC_{\smash{j}}^{-}(S\,|\,K,\Zp) }="1";
(32,0)*+{ \TP_j(S\,|\,K,\Zp) }="2";
{ \ar^-{\varphi} "2";"1";};
\end{xy}$$
is an isomorphism, for $j \geqslant 0$. We evaluate the common group
for $j = 2k > 0$, and consider the cases, where $S$ is of equal
characteristic and of mixed characteristic, separately. 

In the mixed characteristic case, the groups
$\THH_j(S\,|\,K,\mathbb{Z}_p)$ for $j =2k > 0$ are zero
by~\cite[Remark~2.4.2]{hm4}. Hence, the homotopy fixed points
spectral sequence
$$E_{i,j}^2 =
H^{-i}(B\hskip.5pt\mathbb{T},\THH_j(S\,|\,K,\mathbb{Z}_p)) \Rightarrow
\TC_{i+j}^{-}(S\,|\,K,\mathbb{Z}_p)$$
shows that the same is true for the groups
$\TC_j^{-}(S\,|\,K,\mathbb{Z}_p)$.

If $S = k_S[\![\pi]\!]$ is of equal characteristic, then
$\TP_*(S\,|\,K)$ is a graded algebra over the graded ring
$\TP_*(\Fp) = \Zp[v^{\pm1}]$, where $\deg(v) = -2$, and hence, is
$2$-periodic. We have canonical isomorphisms
$$\xymatrix{
{ \TP_0(S\,|\,K) } &
{ \TC_0^{-}(S\,|\,K) } \ar[l]_-{\can} &
{ \TF_0(S\,|\,K;p) } \ar[l] \ar[r] &
{ W(k_S), } \cr
}$$
where the latter follows from Lemma~\ref{lem:wittvectors}. Under this
identification, the two maps $\varphi, \can \colon
\TC_0^{-}(S\,|\,K) \to \TP_0(S\,|\,K)$ are given by the automorphisms
of $W(k_s)$ induced by $\varphi,\id \colon k_S \to k_S$,
respectively. Finally, the maps  
$\varphi,\can \colon \TC_2^{-}(\Fp) \to \TP_2(\Fp)$ are given by
$\varphi(v^{-1}) = v^{-1}$ and $\can(v^{-1}) = pv^{-1}$;
see~\cite[Section~IV.4]{nikolausscholze}. Hence, we find that for
$j = 2k \geqslant 0$, the maps
$$\begin{xy}
(0,0)*+{ \TC_j^{-}(S\,|\,K) }="1";
(28,0)*+{ \TP_j(S\,|\,K) }="2";
{ \ar^-{\varphi, \; \can} "2";"1";};
\end{xy}$$
are respectively an isomorphism and $p^k$ times an isomorphism, so for
$j = 2k > 0$, their difference is an isomorphism, as stated.
\end{proof}

\begin{proof}[Proof of Theorem~B~{\rm(2)}]We consider the diagram of spectra
$$\begin{xy}
(0,7)*+{ \TF(A\,|\,D;p,\Zp) }="11";
(34,7)*+{ \TC^{-}(A\,|\,D,\Zp)\phantom{,} }="12";
(0,-7)*+{ \TF(S\,|\,K;p,\Zp) }="21";
(34,-7)*+{ \TC^{-}(S\,|\,K,\Zp), }="22";
{ \ar "12";"11";};
{ \ar_-{\Ird_{A/S}} "11";"21";};
{ \ar "22";"21";};
{ \ar_-{\Ird_{A/S}} "12";"22";};
\end{xy}$$
where the vertical maps are induced by the equivalence of cyclotomic
spectra from part~(1) of the theorem, and where the horizontal maps
are the maps defined in the proof of
Corollary~\ref{cor:frobeniusequivalence}. The diagram commutes by
naturality of the horizontal maps, the vertical maps are equivalences,
and, by Corollary~\ref{cor:frobeniusequivalence}, the horizontal maps
induce equivalences of connective covers. The class
$y = \Ird_{A/S}(1) \in \TC_0^{-}(A\,|\,D,\Zp)$ defines a component in
the mapping space 
$$\begin{aligned}
{} & \Map_{\operatorname{Mod}_{\THH(S\,|\,K,\Zp)}
(\mathsf{Sp}_p^{B\,\mathbb{T}})} (\THH(S\,|\,K,\Zp),
\THH(A\,|\,D,\Zp)) \cr
{} & \simeq
\Map_{\mathsf{Sp}_p^{B\,\mathbb{T}}}
(\mathbb{S}_p,\THH(A\,|\,D,\Zp))
\simeq \Omega^{\infty}\TC^{-}(A\,|\,D,\Zp) \cr
\end{aligned}$$
and any point in this component is an equivalence. Now, let $I_{A/S}$
be the map induced by extension of scalars along the inclusion of $S$
in $A$. It follows from Remark~\ref{rem:reducedtrace}
and from the construction of the class $y$ that $I_{A/S}(1) = d \cdot
y$. This shows that
$$I_{A/S} \simeq d \cdot \Ird_{A/S}$$
as maps of $\THH(S\,|\,K,\Zp)$-modules in
$\mathsf{Sp}_p^{B\,\mathbb{T}}$. We claim that this implies that also
$$\Tr_{A/S} \simeq d \cdot \Trd_{A/S}$$
as maps of $\THH(S\,|\,K,\Zp)$-modules in
$\mathsf{Sp}_p^{B\,\mathbb{T}}$. Indeed,
$$d \cdot (\Tr_{A/S} \circ \Ird_{A/S}) \simeq \Tr_{A/S} \circ (d \cdot
\Ird_{A/S}) \simeq \Tr_{A/S} \circ I_{A/S} \simeq d^2 \cdot \id,$$
which implies that
$$\Tr_{A/S} \circ \Ird_{A/S} \simeq d \cdot \id,$$
since $\TC_0^{-}(S\,|\,K,\Zp) = W(k_S)$ is an integral domain. But we
also have
$$(d \cdot \Trd_{A/S}) \circ \Ird_{A/S} \simeq d \cdot (\Trd_{A/S}
\circ \Ird_{A/S}) \simeq d \cdot \id,$$
and since $\Ird_{A/S}$ is an equivalence, the claim follows. This
completes the proof.
\end{proof}

\begin{remark}\label{rem:uniquenessoftrd}The component in the mapping
space
$$\Map_{\operatorname{Mod}_{\THH(S\,|\,K,\Zp)}(\mathsf{Sp}_p^{B\,\mathbb{T}})}
(\THH(A\,|\,D,\Zp),\THH(S\,|\,K,\Zp))$$
that contains the equivalence $\Trd_{A/S}$ is uniquely determined by
the property that
$$d \cdot \Trd_{A/S} \simeq \Tr_{A/S}.$$
Indeed, the mapping space in question is simultaneously a torsor for
the $\mathbb{E}_{\infty}$-monoids of endomorphisms of the domain and
target, and their rings of components are both integral domains, being
isomorphic to $W(k_S)$. 
\end{remark}

\begin{proof}[Proof of Theorem~\ref{thm:theoremA}]For all integers $j$,
the equivalence of cyclotomic spectra $\Trd_{A/S}$ in
Theorem~\ref{thm:theoremB}~(1) induces an isomorphism
\vspace{-1mm}
$$\begin{xy}
(0,0)*+{ \TC_j(A\,|\,D,\Zp) }="1";
(33,0)*+{ \TC_j(S\,|\,K,\Zp), }="2";
{ \ar^-{\Trd_{A/S}} "2";"1";};
\end{xy}$$
and for $j \geqslant 1$, this induces the desired isomorphism
\vspace{-1mm}
$$\begin{xy}
(0,0)*+{ K_j(D,\Zp) }="1";
(27,0)*+{ K_j(K,\Zp), }="2";
{ \ar^-{\Nrd_{A/S}} "2";"1";};
\end{xy}$$
by Theorem~\ref{thm:KofD}. We claim that for $j \geqslant 1$, the former
isomorphism is canonical and satisfies
$d \cdot \Trd_{A/S} = \Tr_{A/S}$. Indeed, by
Theorem~\ref{thm:evenisomorphism} we have exact sequences  
$$\begin{xy}
(0,7)*+{ 0 }="11";
(0,-7)*+{ 0 }="21";
(22,7)*+{ \TC_j(A\,|\,D,\Zp) }="12";
(22,-7)*+{ \TC_j(S\,|\,K,\Zp) }="22";
(55,7)*+{ \TC_j^{-}(A\,|\,D,\Zp) }="13";
(55,-7)*+{ \TC_j^{-}(S\,|\,K,\Zp) }="23";
(90,7)*+{ \TP_j(A\,|\,D,\Zp)\phantom{,} }="14";
(90,-7)*+{ \TP_j(S\,|\,K,\Zp), }="24";
{ \ar "12";"11";};
{ \ar^-{i} "13";"12";};
{ \ar^-{\varphi - \can} "14";"13";};
{ \ar^-{\Trd_{A/S}} "22";"12";};
{ \ar^-{\Trd_{A/S}} "23";"13";};
{ \ar^-{\Trd_{A/S}} "24";"14";};
{ \ar "22";"21";};
{ \ar^-{i} "23";"22";};
{ \ar^-{\varphi - \can} "24";"23";};
\end{xy}$$
for $j \geqslant 1$ and odd, and by Theorem~\ref{thm:theoremB}~(2), the
middle and right-hand vertical maps are canonical and satisfy the
desired identity. Theorem~\ref{thm:evenisomorphism} shows similarly
that for $j \geqslant 2$ and even, there are exact sequences
$$\begin{xy}
(0,7)*+{ \TC_{j+1}^{-}(A\,|\,D,\Zp) }="11";
(0,-7)*+{ \TC_{j+1}^{-}(S\,|\,K,\Zp) }="21";
(38,7)*+{ \TP_{j+1}(A\,|\,D,\Zp)\phantom{,} }="12";
(38,-7)*+{ \TP_{j+1}(S\,|\,K,\Zp), }="22";
(73,7)*+{ \TC_j(A\,|\,D,\Zp) }="13";
(73,-7)*+{ \TC_j(S\,|\,K,\Zp) }="23";
(95,7)*+{ 0\phantom{,} }="14";
(95,-7)*+{ 0, }="24";
{ \ar^-{\varphi-\can} "12";"11";};
{ \ar^-{\partial} "13";"12";};
{ \ar "14";"13";};
{ \ar^-{\Trd_{A/S}} "21";"11";};
{ \ar^-{\Trd_{A/S}} "22";"12";};
{ \ar^-{\Trd_{A/S}} "23";"13";};
{ \ar^-{\varphi-\can} "22";"21";};
{ \ar^-{\partial} "23";"22";};
{ \ar "24";"23";};
\end{xy}$$
and by Theorem~\ref{thm:theoremB}~(2), the left-hand and middle vertical
maps are canonical and satisfy the desired identity. This completes
the proof.
\end{proof}

\begin{lemma}\label{lem:pushforward}The map
$i_* \colon \TC_j(k_S,\mathbb{Z}_p) \to \TC_j(S,\mathbb{Z}_p)$ is zero
for all integers $j$.
\end{lemma}

\begin{proof}First, for $j = 0$, we consider the following diagram
with exact rows.
$$\xymatrix{
{ K_0(k_S,\mathbb{Z}_p) } \ar[r]^-{i_*} \ar[d]^-{\tr} &
{ K_0(S,\mathbb{Z}_p) } \ar[r]^-{j^*} \ar[d]^-{\tr} &
{ K_0(K,\mathbb{Z}_p) } \ar[d]^-{\tr} \cr
{ \TC_0(k_S,\mathbb{Z}_p) } \ar[r]^-{i_*} &
{ \TC_0(S,\mathbb{Z}_p) } \ar[r]^-{j^*} &
{ \TC_0(S\,|\,K,\mathbb{Z}_p) } \cr
}$$
In the top row, the map $j^*$ is an isomorphism,
since $S$ and $K$ both are local rings. Therefore, the map $i_*$ in
the top row is zero, and since the left-hand and middle vertical maps
are isomorphisms by~\cite[Theorem~D]{hm}, the map $i_*$ in the bottom
row is zero, as stated. Next, we claim that the map
\vspace{-1mm}
$$\xymatrix{
{ \TC_j(S,\mathbb{Z}_p) } \ar[r]^-{i^{\hskip.5pt *}} &
{ \TC_j(k_S,\mathbb{Z}_p) } \cr
}$$
is surjective. Granting this, the lemma follows. Indeed, if
$x = i^{\hskip.5pt *}(y)$, then 
$$i_*(x) = i_*(x \cdot 1) = i_*(i^{\hskip.5pt *}(y) \cdot 1) = y \cdot
i_*(1),$$
by the projection formula, and we have already proved that
$i_*(1) = 0$. The claims needs proof only if $j = 0$ or $j = -1$,
since the target of the map in question is zero, otherwise, and in
these cases, the domain and target both are free
$\mathbb{Z}_p$-modules of rank $1$. For $j = 0$, the map is a
$\mathbb{Z}_p$-algebra homomorphism, and therefore, it is necessarily
an isomorphism. Finally, for $j = -1$, we consider the diagram
$$\xymatrix{
{ \TP_0(S,\mathbb{Z}_p) } \ar[r]^-{i^{\hskip.5pt *}}
\ar[d]^-{\partial} &
{ \TP_0(k_S,\mathbb{Z}_p) } \ar[d]^-{\partial} \cr
{ \TC_{-1}(S,\mathbb{Z}_p) } \ar[r]^-{i^{\hskip.5pt *}} &
{ \TC_{-1}(k_S,\mathbb{Z}_p). } \cr
}$$
The top horizontal map is canonically identified with the map in the
statement of Lemma~\ref{lem:wittvectors}, and hence, it is
surjective. The right-hand vertical map also is surjective, since
$\TC_*^{-}(k_S,\mathbb{Z}_p)$ is concentrated in even degrees by
B\"{o}kstedt periodicity. It follows that the lower horizontal map is
surjective, as claimed.
\end{proof}

\begin{corollary}\label{cor:tczero}The $\mathbb{Z}_p$-module
$\TC_0(S\,|\,K,\mathbb{Z}_p)$ is free of rank $2$.
\end{corollary}

\begin{proof}By Lemma~\ref{lem:pushforward}, we have a short exact
sequence
\vspace{-1mm}
$$\xymatrix{
{ 0 } \ar[r] &
{ \TC_0(S,\mathbb{Z}_p) } \ar[r]^-{j^*} &
{ \TC_0(S\,|\,K,\mathbb{Z}_p) } \ar[r]^-{\delta} &
{ \TC_{-1}(k_S,\mathbb{Z}_p) } \ar[r] &
{ 0, } \cr
}$$
and, we have as already remarked, the left-hand term and the
right-hand term both are free $\mathbb{Z}_p$-modules of rank $1$.
\end{proof}

\begin{proof}[Proof of Theorem~\ref{thm:theoremB}~{\rm (3)}]We consider
the diagram
$$\xymatrix{
{ 0 } \ar[r] &
{ \TP_1(A\,|\,D,\Zp)_{\varphi} } \ar[r]^-{\partial}
\ar[d]^-{\Trd_{A/S}} &
{ \TC_0(A\,|\,D,\Zp) } \ar[r]^-{i} \ar[d]^-{\Trd_{A/S}} &
{ \TC_0^{-}(A\,|\,D,\Zp)^{\varphi} } \ar[r] \ar[d]^-{\Trd_{A/S}} &
{ 0\phantom{,} } \cr
{ 0 } \ar[r] &
{ \TP_1(S\,|\,K,\Zp)_{\varphi} } \ar[r]^-{\partial} &
{ \TC_0(S\,|\,K,\Zp) } \ar[r]^-{i} &
{ \TC_0^{-}(S\,|\,K,\Zp)^{\varphi} } \ar[r] &
{ 0, } \cr
}$$
where $(-)^{\varphi}$ and $(-)_{\varphi}$ indicate the kernel and
cokernel of $\varphi - \can$, respectively. The vertical maps are
induced by the equivalence of cyclotomic spectra from part~(1) of the
theorem. Accordingly, the vertical maps are all isomorphisms and the
diagram commutes. By Corollary~\ref{cor:tczero}, the middle terms are
free $\mathbb{Z}_p$-modules of rank $2$, and we claim that the
right-hand terms are free $\mathbb{Z}_p$-modules of rank $1$. Indeed,
we have already identified $\TC_0^{-}(S\,|\,K,\Zp)$ and
$\TP_0(S\,|\,K,\Zp)$ with $W(k_S)$ and the maps $\varphi$ and $\can$
with the maps induced by $\varphi,\id \colon k_S \to k_S$, and since
$$\begin{xy}
(0,0)*+{ W(\Fp) }="1";
(22,0)*+{ W(k_S) }="2";
(44,0)*+{ W(k_S) }="3";
{ \ar "2";"1";};
{ \ar@<.7ex>^-{W(\varphi)} "3";"2";};
{ \ar@<-.7ex>_-{\id} "3";"2";};
\end{xy}$$
is an equalizer of rings, the claim follows. We conclude that the
left-hand terms are free $\mathbb{Z}_p$-modules of rank $1$ as well. 

We choose a basis
$(e_1,e_2)$ of $\TC_0(S\,|\,K,\Zp)$ such that $e_1$ is in the image of
$\partial$ and such that $i(e_2) = 1$, and let $(e_1',e_2')$ be the basis
$\smash{ (\Trd_{A/S}^{-1}(e_1),\Trd_{A/S}^{-1}(e_2)) }$ of
$\TC_0(A\,|\,D,\Zp)$. By part~(2) of the theorem, the matrix
that represents 
$$\begin{xy}
(0,0)*+{ \TC_0(A\,|\,D,\Zp) }="1";
(32,0)*+{ \TC_0(S\,|\,K,\Zp) }="2";
{ \ar^-{\Tr_{A/S}} "2";"1";};
\end{xy}$$
with respect to these bases is of the form
$$\begin{pmatrix}
d & \;\;a \cr
0 & \;\;d \cr
\end{pmatrix} \in M_2(\Zp).$$
We proceed to show that if $p$ divides $d$, then $d$ does not divide
$a$, which proves part~(3) of the theorem. To this end, we consider
the commutative diagram
$$\begin{xy}
(0,7)*+{ K_0(D,\Zp) }="11";
(29,7)*+{ \TC_0(A\,|\,D,\Zp) }="12";
(63,7)*+{ \TC_0^{-}(A\,|\,D,\Zp)^{\varphi}\phantom{.} }="13";
(0,-7)*+{ K_0(K,\Zp) }="21";
(29,-7)*+{ \TC_0(S\,|\,K,\Zp) }="22";
(63,-7)*+{ \TC_0^{-}(S\,|\,K,\Zp)^{\varphi}. }="23";
{ \ar^-{\tr} "12";"11";};
{ \ar^-{i} "13";"12";};
{ \ar^-{\tr} "22";"21";};
{ \ar^-{i} "23";"22";};
{ \ar^-{N_{D/K}} "21";"11";};
{ \ar^-{\Tr_{A/S}} "22";"12";};
{ \ar^-{\Tr_{A/S}} "23";"13";};
\end{xy}$$
The left-hand terms are both free $\Zp$-modules of rank $1$ generated
by the classes $[D]$ and $[K]$, and $N_{D/K}([D]) = d^2 \cdot
[K]$. Moreover, the lower horizontal maps are ring homomorphisms, and
hence, map the element $1 = [K]$ of the lower left-hand term to the
element $1$ in the lower right-hand term. Since part~(2) of the theorem
shows that the right-hand vertical map is equal to $d \cdot
\Trd_{A/S}$, we conclude that coordinates of the class $\tr([D]) \in
\TC_0(A\,|\,D,\Zp)$ with respect to the basis $(e_1',e_2')$ take the
form
$$\begin{pmatrix}
\, b \, \cr
\, d \, \cr
\end{pmatrix} \in M_{2,1}(\Zp).$$
Hence, the coordinates of $\Tr_{A/S}(\tr([D]))$ with respect to the
basis $(e_1,e_2)$ are
$$\begin{pmatrix}
d & \;\; a \cr
0 & \;\; d \cr
\end{pmatrix} \begin{pmatrix}
\, b \, \cr
\, d \, \cr
\end{pmatrix} = \begin{pmatrix}
d(a+b) \cr
d^2 \cr
\end{pmatrix}.$$
It follows that $a+b \in \Zp$ is divisible by $d$, since
$$\Tr_{A/S}(\tr([D])) = \tr(N_{D/K}([D])
= d^2 \cdot \tr([K]) = d^2 \cdot 1.$$
Therefore, $d$ divides $a$ if and only if $d$ divides
$b$. We now consider the following commutative diagram with exact
rows.
$$\xymatrix{
{ K_0(A,\Zp) } \ar[r]^-{j^*} \ar[d]^{\tr} &
{ K_0(D,\Zp) } \ar[r]^-{\delta} \ar[d]^-{\tr} &
{ K_{-1}(k_T,\mathbb{Z}_p) } \ar[d]^-{\tr} \cr
{ \TC_0(A,\Zp) } \ar[r]^-{j^*} &
{ \TC_0(A\,|\,D,\Zp) } \ar[r]^-{\delta} &
{ \TC_{-1}(k_T,\Zp) } \cr
}$$
Since $A$ and $D$ are local rings, the map $j^*$ in
the top row is an isomorphism between free $\mathbb{Z}_p$-modules of
rank $1$. Moreover, the left-hand vertical map is an isomorphism
by~\cite[Theorem~D]{hm}, since $A$ is a finite
$\mathbb{Z}_p$-algebra. Hence, in the bottom row, the middle term is a
free $\mathbb{Z}_p$-module of rank $2$, and the left-hand term and
the right-hand term both are free $\mathbb{Z}_p$-modules of rank
$1$. Therefore, in the bottom row, the map $j^*$ is injective, and its
cokernel is a free $\mathbb{Z}_p$-module of rank $1$. We conclude
that
$$\tr([D]) = \tr(j^*([A])) = j^*(\tr([A])) \in \TC_0(A\,|\,D,\Zp)$$
is not divisible by $p$. So if $p$ divides $d$, then $d$ does not
divide $b$, and hence, $d$ does not divide $a$, as we wished to
prove.
\end{proof}

\begin{corollary}\label{cor:nontrivialaction}If $p$ divides $d$, then
the $G$-action on $\THH(A\,|\,D)$ is non-trivial. 
\end{corollary}

\begin{proof}If the $G$-action on $\THH(A\,|\,D)$ were trivial, then
the diagram
$$\xymatrix{
{ \THH(A\,|\,D) } \ar[r]^-{f^*} &
{ \THH(A\otimes_ST\,|\,D\otimes_KL)^{hG} } \ar@<-.7ex>[d]_-{\Trd_{A
    \otimes_ST/T}} \cr
{ \THH(S\,|\,K) } \ar[r]^-{f^*} &
{ \THH(T\,|\,L)^{hG} } \ar@<-.7ex>[u]_-{\Ird_{A \otimes_ST/T}} \cr
}$$
would be meaningful and would define an equivalence
\vspace{-1mm}
$$\begin{xy}
(0,0)*+{ \THH(A\,|\,D) }="1";
(27,0)*+{ \THH(S\,|\,K) }="2";
{ \ar^-{\Trd_{A/S}} "2";"1";};
\end{xy}$$
of $\THH(S\,|\,K)$-modules in cyclotomic spectra such that
$d \cdot \Trd_{A/S} \simeq \Tr_{A/S}$. However, if $p$ divides $d$,
then Theorem~\ref{thm:theoremB}~(3) shows that such an equivalence does
not exist, so the $G$-action on $\THH(A\,|\,D)$ is necessarily
non-trivial.
\end{proof}

\begin{remark}\label{rem:valuationrings}The proof of
Theorem~\ref{thm:trreducedtrace} gives maps of exact sequences
$$\begin{xy}
(0,7)*+{ \cdots }="11";
(21,7)*+{ \THH_j(k_T) }="12";
(47,7)*+{ \THH_j(A) }="13";
(74.5,7)*+{ \THH_j(A\,|\,D) }="14";
(96.5,7)*+{ \cdots }="15";
(0,-7)*+{ \cdots }="21";
(21,-7)*+{ \THH_j(k_S) }="22";
(47,-7)*+{ \THH_j(S) }="23";
(74.5,-7)*+{ \THH_j(S\,|\,K) }="24";
(96.5,-7)*+{ \cdots }="25";
{ \ar^-{\delta} "12";"11";};
{ \ar^-{i_*} "13";"12";};
{ \ar^-{j^*} "14";"13";};
{ \ar^-{\delta} "15";"14";};
{ \ar@<-.7ex>_-{\Trd_{A/S}} "22";"12";};
{ \ar@<-.7ex>_-{\Ird_{A/S}} "12";"22";};
{ \ar@<-.7ex>_-{\Trd_{A/S}} "23";"13";};
{ \ar@<-.7ex>_-{\Ird_{A/S}} "13";"23";};
{ \ar@<-.7ex>_-{\Trd_{A/S}} "24";"14";};
{ \ar@<-.7ex>_-{\Ird_{A/S}} "14";"24";};
{ \ar^-{\delta} "22";"21";};
{ \ar^-{i_*} "23";"22";};
{ \ar^-{j^*} "24";"23";};
{ \ar^-{\delta} "25";"24";};
\end{xy}$$
such that all maps are $\THH_*(S)$-linear and such that the right-hand
vertical maps are isomorphisms defined in this paper. The analogous
statement holds for the groups $\TR_j^n$, $\TR_j$, $\TF_j$, and for
the respective groups with $p$-adic coefficients. We further expect
that the left-hand vertical maps in the diagram agree with the maps
$\Tr_{k_T/k_S}$ and $I_{k_T/k_S}$, respectively. Indeed, this would
explain the appearance of the kernel of $\Tr_{k_T/k_S}$
in the calculations of $\HH_*(A/\Zp)$ by the second author
in~\cite[Theorem~3.5]{larsen} and of $\THH_*(A,\Zp)$ by the third
author and Chan in~\cite[Theorem~5.1]{chanlindenstrauss}.
\end{remark}

\providecommand{\bysame}{\leavevmode\hbox to3em{\hrulefill}\thinspace}
\providecommand{\MR}{\relax\ifhmode\unskip\space\fi MR }
\providecommand{\MRhref}[2]{%
  \href{http://www.ams.org/mathscinet-getitem?mr=#1}{#2}
}
\providecommand{\href}[2]{#2}

\end{document}